\documentclass{amsart}

\usepackage{amsmath}
\usepackage{amsfonts}
\usepackage{amssymb,enumerate}
\usepackage{amsthm}
\usepackage[all]{xy}
\usepackage{hyperref}

\newtheorem{lem}{Lemma}[section]
\newtheorem{cor}[lem]{Corollary}
\newtheorem{prop}[lem]{Proposition}
\newtheorem{thm}[lem]{Theorem}

\newtheorem{Defn}[lem]{Definition}
\newtheorem{Ex}[lem]{Example}
\newtheorem{Question}[lem]{Question}
\newtheorem{Property}[lem]{Property}
\newtheorem{Properties}[lem]{Properties}
\newtheorem{Discussion}[lem]{Remark}
\newtheorem{Construction}[lem]{Construction}
\newtheorem{Notation}[lem]{Notation}
\newtheorem{Fact}[lem]{Fact}
\newtheorem{Notationdefinition}[lem]{Definition/Notation}
\newtheorem{Remarkdefinition}[lem]{Remark/Definition}
\newtheorem{Subprops}{}[lem]
\newtheorem{Para}[lem]{}

\newenvironment{defn}{\begin{Defn}\rm}{\end{Defn}}
\newenvironment{ex}{\begin{Ex}\rm}{\end{Ex}}

\newenvironment{properties}{\begin{Properties}\rm}{\end{Properties}}

\newenvironment{fact}{\begin{Fact}\rm}{\end{Fact}}

\newenvironment{subprops}{\begin{Subprops}\rm}{\end{Subprops}}
\newenvironment{para}{\begin{Para}\rm}{\end{Para}}
\newenvironment{disc}{\begin{Discussion}\rm}{\end{Discussion}}

\newtheorem{intthm}{Theorem}

\numberwithin{equation}{lem}

\newcommand{\cat}[1]{\mathcal{#1}}

\newcommand{\catc}{\cat{C}}

\newcommand{\pd}{\operatorname{pd}}

\newcommand{\id}{\operatorname{id}}

\newcommand{\rank}{\operatorname{rank}}

\newcommand{\ann}{\operatorname{Ann}}

\newcommand{\Span}{\operatorname{Span}}

\newcommand{\ch}{\operatorname{char}}

\newcommand{\HH}{\operatorname{H}}

\newcommand{\coker}{\operatorname{Coker}}
\newcommand{\spec}{\operatorname{Spec}}

\newcommand{\im}{\operatorname{Im}}
\newcommand{\shift}{\mathsf{\Sigma}}

\newcommand{\Ker}{\operatorname{Ker}}

\newcommand{\ideal}[1]{\mathfrak{#1}}
\newcommand{\m}{\ideal{m}}
\newcommand{\n}{\ideal{n}}
\newcommand{\p}{\ideal{p}}

\newcommand{\ol}{\overline}
\newcommand{\wti}{\widetilde}

\newcommand{\ass}{\operatorname{Ass}}
\newcommand{\supp}{\operatorname{Supp}}

\newcommand{\zz}{\mathbb{Z}}

\newcommand{\xra}{\xrightarrow}

\newcommand{\res}{\xra{\simeq}}

\newcommand{\vf}{\varphi}

\newcommand{\alf}[1]{\alpha^{#1}}
\newcommand{\jinj}[1]{j^{#1}}
\newcommand{\ppro}[1]{p^{#1}}
\newcommand{\iinj}[1]{i^{#1}}
\newcommand{\qpro}[1]{q^{#1}}

\newcommand{\opro}[1]{\omega^{#1}}

\renewcommand{\geq}{\geqslant}
\renewcommand{\leq}{\leqslant}
\renewcommand{\ker}{\Ker}

\newcommand{\s}{\mathrm{S}^2}
\newcommand{\sw}{\mathrm{s}^2}

\begin{document}

\bibliographystyle{amsplain}

\title{Second symmetric powers of  chain complexes}

\author[A.\ J.\ Frankild]{Anders J.~Frankild}
\thanks{This work was completed after the untimely passing of Anders J.\ Frankild
on 10 June 2007.}

\author[S.\ Sather-Wagstaff]{Sean Sather-Wagstaff}
\address{Sean Sather-Wagstaff, Department of Mathematics,
NDSU Dept \# 2750,
PO Box 6050,
Fargo, ND 58108-6050
USA}
\email{Sean.Sather-Wagstaff@ndsu.edu}
\urladdr{http://www.ndsu.edu/pubweb/\~{}ssatherw/}

\author[A.\ Taylor]{Amelia Taylor}
\address{Amelia Taylor, Colorado College
14 E. Cache La Poudre St.
Colorado Springs, CO 80903, USA}
\email{amelia.taylor@coloradocollege.edu}
\urladdr{http://faculty1.coloradocollege.edu/~ataylor}


\begin{abstract}
We investigate 
Buchbaum and Eisenbud's construction of
the second symmetric power $\s_R(X)$ of a chain complex
$X$ of modules over a commutative ring $R$. 
We state and prove a number of results from the folklore of the subject
for which we know of no good direct references.
We also provide several explicit computations
and examples.
We use this construction to prove the following version of a 
result of Avramov, Buchweitz, and \c{S}ega:
Let $R\to S$ be a module-finite ring homomorphism such that $R$ is
noetherian and local, and such that 2  is a unit in $R$. Let $X$ be a complex of finite rank
free $S$-modules such that $X_n=0$ for each $n<0$. If
$\cup_n\ass_R(\HH_n(X\otimes_SX))\subseteq\ass(R)$ and if
$X_{\p}\simeq S_{\p}$ for each $\p\in\ass(R)$, then $X\simeq S$.
\end{abstract}

\maketitle

\section*{Introduction}

Multilinear constructions like tensor products and symmetric powers 
are important tools for
studying modules over commutative rings. 
In recent years, these notions have been extended to the realm
of chain complexes of $R$-modules. 
(Consult Section~\ref{sec01}
for background information on complexes.)
For instance, Buchsbaum and Eisenbud's
description~\cite{buchsbaum:asffr} of the minimal free resolutions of Gorenstein ideals of grade 3
uses the second symmetric power of a certain free resolution.

In this paper, we investigate  Buchsbaum and Eisenbud's
second symmetric power functor: 
for a chain complex $X$ of modules over a commutative ring $R$, we set
$\s_R(X)=(X\otimes_R X)/(Y+Z)$
where $Y$ is the graded submodule generated by all elements of the form
$x\otimes
x'-(-1)^{|x||x'|}x'\otimes x$
and  $Z$ is the graded submodule generated by all elements of the form
$x\otimes x$ where $x$ has odd
degree.\footnote{Note that the definition of $\s_R(X)$ given
in~\cite{buchsbaum:asffr}  does
not yield the complex described on~\cite[p.~452]{buchsbaum:asffr} unless 2 is a unit in $R$. 
The corrected definition can be found, for instance,
in~\cite[(3.4.3)]{bruns:cmr}.} 

Our main result is the following 
version of a result of 
Avramov, Buchweitz and \c{S}ega~\cite[(2.2)]{avramov:edcrcvct}
for complexes. 
It is motivated by our work in~\cite{frankild:rbsc} extending
the results of~\cite{avramov:edcrcvct}. 
Note
that $\s_R(X)$ does not appear in the statement of
Theorem~\ref{intthmC}; however, it is the key tool for the proof,
given
in~\ref{thmCpf}. 

\begin{intthm} \label{intthmC}
Let $R\to S$ be a module-finite ring homomorphism such that $R$ is
noetherian and local, and such that 2  is a unit in $R$. Let $X$ be a complex of finite rank
free $S$-modules such that $X_n=0$ for each $n<0$. If
$\cup_n\ass_R(\HH_n(X\otimes_SX))\subseteq\ass(R)$ and if
$X_{\p}\simeq S_{\p}$ for each $\p\in\ass(R)$, then $X\simeq S$.
\end{intthm}

Much of this paper is devoted to statements and proofs of results
from the folklore of this subject. 
Section~\ref{sec02}  contains basic
properties of $\s_R(X)$, most of which 
are motivated by the behavior of tensor products of complexes and the
properties of symmetric powers of modules.  
This section ends with an explicit description of the modules occuring
in $\s_R(X)$; see Theorem~\ref{betti01}.
Section~\ref{sec03}  examines the homological properties of $\s_R(X)$,
and includes the proof of Thoerem~\ref{intthmC}.
The paper concludes with Section~\ref{sec05}, which is devoted to 
explicit computations.

\section{Complexes} \label{sec01}

Throughout this paper $R$ and $S$ are commutative rings with identity.
The term ``module'' is short for ``unital module''.

This section consists of definitions, notation and background information for
use in the remainder of the paper.

\begin{defn} \label{notn01}
An \emph{$R$-complex} is a sequence of
$R$-module homomorphisms
$$X =\cdots\xra{\partial^X_{n+1}}X_n\xra{\partial^X_n}
X_{n-1}\xra{\partial^X_{n-1}}\cdots$$
such that $ \partial^X_{n-1}\partial^X_{n}=0$ for each integer $n$.
A complex $X$ is \emph{degreewise-finite} if each $X_n$ is finitely generated;
it is
\emph{bounded-below} if $X_n=0$ for $n\ll 0$. 

The
$n$th \emph{homology module} of $X$ is
$\HH_n(X):=\Ker(\partial^X_{n})/\im(\partial^X_{n+1})$.
The \emph{infimum} of 
$X$ is
$\inf(X):=\inf\{i\in\zz\mid\HH_n(X)\neq 0\}$,
and the \emph{large support} of $X$ is
$$\supp_R(X)=\{\p\in\spec(R)\mid X_{\p}\not\simeq 0\}
=\cup_n\supp_R(\HH_n(X)).$$
For each $x\in X_n$, we set $|x|:=n$. An $R$-complex $X$ is
\emph{homologically degreewise-finite} if 
$\HH_n(X)$ is finitely generated for each $n$; it is
\emph{homologically finite} if the $R$-module
$\oplus_{n\in\zz}\HH_n(X)$ is finitely generated. 

For each integer $i$,
the $i$th \emph{suspension} (or \emph{shift}) of $X$, denoted
$\shift^i X$, is the complex with $(\shift^i X)_n=X_{n-i}$ and
$\partial_n^{\shift^i X}=(-1)^i\partial_{n-i}^X$. The notation
$\shift X$ is short for $\shift^1 X$.
\end{defn}

\begin{defn} \label{notn02}
Let $X$ and $Y$ be $R$-complexes.
A \emph{morphism} from $X$ to $Y$ is a sequence 
of $R$-module homomorphisms
$\{f_n\colon
X_n\to Y_n\}$ such that $f_{n-1}\partial^X_n=\partial^Y_n f_n$ for
each $n$. A morphism of complexes $\alpha\colon X\to Y$ induces
homomorphisms on homology modules
$\HH_n(\alpha)\colon\HH_n(X)\to\HH_n(Y)$, and $\alpha$ is a
\emph{quasiisomorphism} when each $\HH_n(\alpha)$ is bijective.
Quasiisomorphisms are designated by the symbol ``$\simeq$''. 
\end{defn}

\begin{defn} \label{notn02'}
Let $X$ and $Y$ be $R$-complexes. Two morphisms $f,g\colon X\to Y$ are
\emph{homotopic} if  there exists a sequence of 
homomorphisms $s=\{s_n\colon X_n\to Y_{n+1}\}$ 
such that $f_n=g_n+\partial^Y_{n+1}
s_n+s_{n-1}\partial^X_n$ for each $n$; here we say that $s$ is a
\emph{homotopy} from $f$ to $g$. The morphism $f$ is a
\emph{homotopy equivalence} if there is a morphism $h\colon Y\to X$
such that the compositions $fh$ and $hf$ are homotopic to the
respective identity morphisms $\id_Y$ and $\id_X$, and then $f$ and
$h$ are \emph{homotopy inverses}.
\end{defn} 

\begin{defn} \label{notn03}
Given two bounded-below complexes $P$ and $Q$ of projective $R$-modules,
we write $P\simeq Q$ when there is a quasiisomorphism $P\res Q$.
\end{defn}

\begin{fact} \label{fact0001}
The relation $\simeq$ from Definition~\ref{notn03} is an equivalence relation;
see~\cite[(2.8.8.2.2')]{avramov:dgha}
or~\cite[(6.6.ii)]{felix:rht} or~\cite[(6.21)]{foxby:hacr}. 

Let $P$ and $Q$ be bounded-below complexes of projective $R$-modules.
Then any quasiisomorphism $P\res Q$ is a homotopy equivalence;
see~\cite[(1.8.5.3)]{avramov:dgha}
or~\cite[(6.4.iii)]{felix:rht}. (Conversely, it is straightforward to show that
any homotopy equivalence between $R$-complexes is a quasiisomorphism.)
\end{fact} 

\begin{defn} \label{notn03'}
Let $X$ be a homologically bounded-below $R$-complex.  A
\emph{projective (or free) resolution}
of $X$ is a quasiisomorphism $P\xra{\simeq}X$ such that
each $P_n$ is projective (or free) and $P$ is bounded-below;
the resolution $P\res X$ is \emph{degreewise-finite} if  $P$ is
degreewise-finite.
We say that $X$ has \emph{finite projective dimension}
when it admits a projective resolution $P\xra{\simeq}X$ such that
$P_n=0$ for $n\gg 0$.
\end{defn}

\begin{fact} \label{fact0001'}
Let $X$ be a homologically bounded-below $R$-complex.
Then $X$ has a  free
resolution $P\xra{\simeq}X$ such that 
$P_n=0$ for all $n<\inf(X)$; see~\cite[(2.11.3.4)]{avramov:dgha}
or~\cite[(6.6.i)]{felix:rht} or~\cite[(2.6.P)]{foxby:hacr}. 
(It follows  that $P_{\inf(X)}\neq 0$.)
If $P\xra{\simeq}X$ and $Q\xra{\simeq}X$ are projective resolutions of $X$,
then there is a homotopy equivalence $P\xra{\simeq}Q$; 
see~\cite[(6.6.ii)]{felix:rht} or~\cite[(6.21)]{foxby:hacr}.  
If  $R$ is noetherian  and $X$ is homologically degreewise-finite,
then $P$ may be chosen degreewise-finite; see~\cite[(2.11.3.3)]{avramov:dgha}
or~\cite[(2.6.L)]{foxby:hacr}.
\end{fact} 

\begin{defn} \label{notn03''}
Let $X$ be an  $R$-complex that is homologically both bounded-below
and degreewise-finite.  
Assume  that $R$ is noetherian and local with maximal ideal $\m$.
A projective resolution $P\res X$ is \emph{minimal} if 
the complex $P$ is minimal, that is, if
$\im(\partial^P_n)\subseteq\m P_{n-1}$ for each $n$.
\end{defn}

\begin{fact} \label{fact0001''}
Let $X$ be an  $R$-complex that is homologically both bounded-below
and degreewise-finite.  
Assume  that $R$ is noetherian and local with maximal ideal $\m$.
Then $X$ has a minimal free
resolution $P\xra{\simeq}X$ such that 
$P_n=0$ for all $n<\inf(X)$; see~\cite[Prop.\ 2]{apassov:cdgm}
or~\cite[(2.12.5.2.1)]{avramov:dgha}. 
Let $P\xra{\simeq}X$ and $Q\xra{\simeq}X$ be projective resolutions of $X$.
If $P$ is minimal, then there is a
bounded-below
exact complex $P'$ of projective $R$-modules such that
$Q\cong P\oplus P'$; see~\cite[(2.12.5.2.3)]{avramov:dgha}.
It follows that $X$ has finite projective dimension if and only if 
every minimal projective resolution of $X$ is bounded.
It also follows that, if  $P$ and $Q$ are both minimal,
then $P\cong Q$; see~\cite[(2.12.5.2.2)]{avramov:dgha}. 
\end{fact} 

\begin{defn} \label{tp03}
Let $X$ and $Y$ be $R$-complexes.
The  $R$-complex  $X\otimes_RY$ is 
$$\textstyle(X\otimes_RY)_n=\bigoplus_pX_p\otimes_RY_{n-p}$$
with $n$th differential $\partial_n^{X\otimes_RY}$ given on generators by
$$x\otimes y\mapsto \partial^{X}_{|x|}(x)\otimes y +(-1)^{|x|}x\otimes \partial^Y_{|y|}(y).$$
Fix two more $R$-complexes $X',Y'$ and morphisms
$f\colon X\to X'$ and $g\colon Y\to Y'$. 
Define the tensor product
$f\otimes_R g\colon X\otimes_R Y\to X'\otimes_R Y'$
on generators as
$$x\otimes y
\mapsto f_{|x|}(x)\otimes g_{|y|}(y).
$$
One checks readily that
$f\otimes_R g$ is a morphism.
\end{defn} 

\begin{fact} \label{fact01}
Let $P$ and $Q$ be bounded-below 
complexes of projective $R$-modules. 
If $f\colon X\res Y$ is a quasiisomorphism, then so are
$f\otimes_RQ\colon X\otimes_RQ\to Y\otimes_RQ$ and
$P\otimes_Rf\colon P\otimes_RX\to P\otimes_RY$;
see~\cite[(1.10.4.2.2')]{avramov:dgha}
or~\cite[(6.10)]{felix:rht} or~\cite[(7.8)]{foxby:hacr}. 
In particular, if $g\colon P\res Q$ is a quasiisomorphism,
then so is $g\otimes g\colon P \otimes_RP\to Q \otimes_R Q$;
see~\cite[(6.10)]{felix:rht}.
This can be used to show the following facts from~\cite[(7.28)]{foxby:hacr}:
\begin{align*}
\inf(P\otimes_R Q)
&\geq\inf(P)+\inf(Q) \\
\HH_{\inf(P)+\inf(Q)}^R(P\otimes_RQ)
&\cong\HH_{\inf(P)}(P)\otimes_R\HH_{\inf(Q)}(Q).
\end{align*} 

Assume that $R$ is noetherian and that $P$ and $Q$ are homologically degreewise-finite.  
One can use degreewise-finite projective resolutions of
$P$ and $Q$ in order to show that each $R$-module $\HH_n(P\otimes_R
Q)$ is finitely generated;
see~\cite[(7.31)]{foxby:hacr}.  In particular, if $R$ is local,
Nakayama's Lemma conspires with the previous display to produce the
equality $\inf(P\otimes_R Q)=\inf(P)+\inf(Q)$;
see~\cite[(7.28)]{foxby:hacr}. 
\end{fact}

The following technical lemma about power series is used in the proofs
of  Theorem~\ref{symm07''} and
Corollary~\ref{s2fpd02}. 

\begin{lem}\label{poinc01}
Let $Q(t) = \sum_{i=0}^{\infty} r_it^i$ be a power series with
nonnegative integer coefficients, and assume $r_0> 0$. If 
either $Q(t)^2 + Q(-t^2)$ or $Q(t)^2 - Q(-t^2)$ is a non-negative integer, then $r_i = 0$ for all
$i>0$.  Furthermore, 
\begin{enumerate}[\quad\rm(a)]
\item  \label{poinc01item1} $Q(t)^2 + Q(-t^2) \neq 0$;
\item \label{poinc01item2} If $Q(t)^2 - Q(-t^2) = 0$, then $Q(t) = 1$;  
\item \label{poinc01item3} If $Q(t)^2 + Q(-t^2) = 2$, then $Q(t)= 1$; and
\item \label{poinc01item4} If $Q(t)^2 - Q(-t^2) = 2$, then $Q(t) = 2$. 
\end{enumerate}
\end{lem} 

\begin{proof}
We begin by  showing that $r_n=0$ for each $n\geq 1$, by induction on $n$. 
The coefficients of $Q(-t^2)$ in odd degree are all 0.
Hence, the 
degree 1 coefficient of $Q(t)^2 \pm Q(-t^2)$ is
$$0=r_1r_0+r_0r_1=2r_1r_0.$$
It follows that $r_1=0$, since $r_0 > 0$.
Inductively, assume that $n\geq 1$ and that $r_i=0$ for each $i=1,\ldots,n$.
Since the degree $n+1$ coefficient of $Q^R_X(-t^2)$ is either
$\pm r_{\frac{n+1}{2}}$ (when $n+1$ is even) or $0$ (when $n+1$ is odd), 
the induction hypothesis
implies that this coefficient is 0.
The degree $n+1$ coefficient of $Q(t)^2\pm Q(-t^2)$ is
$$0=r_{n+1}r_0+\underbrace{r_nr_1+\cdots+r_1r_n}_{=0}+r_0r_{n+1}=2r_{n+1}r_0$$
and so $r_{n+1}=0$.

The previous paragraph shows that $Q(t)=r_0$, and so $Q(t)^2 \pm Q(-t^2) =r_0^2\mp r_0$.
The conclusions in \eqref{poinc01item1}--\eqref{poinc01item4} 
follow readily, using the assumption $r_0>0$.
\end{proof}

\section{Definition and Basic Properties of $\s_R(X)$} \label{sec02}

We begin this section with the definition of the second symmetric power of a complex.
It is modeled 
on the definition for modules.

\begin{defn} \label{symm01}
Let $X$ be an $R$-complex and let
$\alf X\colon X\otimes_R X\to X\otimes_R X$ be the morphism
described on generators by the formula
$$x\otimes x'\mapsto x\otimes x'-(-1)^{|x||x'|}x'\otimes x.$$
The \emph{weak second symmetric power} of $X$ is defined as
$\sw_R(X):=\coker(\alf X)$.
The \emph{second symmetric power} of $X$ is defined as
$\s_R(X):=\sw_R(X)/\left\langle \ol{x\otimes x}\mid\text{$|x|$ is odd}\right\rangle$.
For each $i\in\mathbb Z$, let $\opro X_i\colon\sw_R(X)_i\to\s_R(X)_i$ be the natural surjection.
\end{defn}

\begin{disc} \label{disc1222}
Let $X$ be an $R$-complex. 
Since $\sw_R(X)$ is defined as a cokernel of a morphism, it is an $R$-complex.
Also, for each $n\in\mathbb Z$ and $x\in X_{2n+1}$,
one has
$$\partial^{X\otimes_RX}_{4n+2}(x\otimes x)=\alf X_{4n+1}(\partial^X_{2n+1}(x)\otimes x).$$
It follows that $\s_R(X)$ is an $R$-complex, and that the sequence
$\{\opro X_i\}$ describes a morphism $\opro X\colon\sw_R(X)\to\s_R(X)$.
\end{disc}

Here are  computations for later use. 
Section~\ref{sec05} contains more involved examples.

\begin{ex}
\label{symm04item4} If $M$ is an $R$-module, then computing
$\s_R(M)$ and $\sw_R(M)$ as complexes 
(considering $M$ as a complex concentrated in degree 0) and 
computing $\s_R(M)$ as a module give the same result. In
particular, we have $\s_R(0)=0=\sw_R(0)$ and $\s_R(R)\cong R\cong\sw_R(R)$.
\end{ex}

\begin{ex}\label{symm045}
For $0\neq x,y\in \shift R$ we have  
$\alf{\shift R}(x\otimes y) =  x\otimes y+ y\otimes x$.
Hence, the natural tensor-cancellation isomorphism $R\otimes_R R\xra\cong R$ 
yields the vertical isomorphisms in the following
commutative diagram:
$$\xymatrix{
(\shift R)\otimes_R(\shift R) \ar[r]^-{\alf{\shift R}} \ar[d]_{\cong}^{\beta}
& (\shift R)\otimes_R(\shift R) \ar[d]_{\cong}^{\beta} \ar[r]^-{\ppro{\shift R}}
&\sw_R(\shift R)\ar[d]_{\cong}^{\ol\beta}
\\
\shift^{2}R \ar[r]^-{(2)} & \shift^{2}R\ar[r]&\shift^2 R/(2)
}
$$
It follows that $\sw_R(\shift R) \cong\shift^2 R/(2)$.

By definition,
the kernel of the natural map
$\opro X\colon\sw_R(X)\to\s_R(X)$
is generated by $\ol{1\otimes 1}\in\sw_R(\shift R)_2$.
Since we have $\ol\beta_2\left(\ol{1\otimes 1}\right)=\ol 1$,
it follows that $\s_R(\shift R)=0$.

More generally, we have $\sw_R(\shift^{2n+1} R) \cong \shift^{4n+2}R/(2)$
and $\s_R(\shift^{2n+1} R)=0$
for each integer $n$.  In particular, if $2R\neq 0$, then 
$$\sw_R(\shift^{2n+1} R) \cong \shift^{4n+2}R/(2)\not\simeq \shift^{4n+2}R \cong\shift^{4n+2}\sw_R(R).$$
Contrast this with the behavior of $\sw_R(\shift^{2n}X)$
and $\s_R(\shift^{2n}X)$ documented in~\eqref{symm04item3}.
\end{ex}

The following properties are straightforward to verify and will be
used frequently in the sequel.

\begin{properties} \label{symm04}
Let $X$ be an $R$-complex.
\begin{subprops}
\label{symm04item5} If  2 is a unit in $R$, then the natural morphism
$\opro X\colon\sw_R(X)\to\s_R(X)$ is an isomorphism, and the morphism $\frac{1}{2}\alf X$ is idempotent.
\end{subprops}
\begin{subprops}
\label{symm04item3}
For each integer $n$, there is a commutative diagram
$$\xymatrix{
(\shift^{2n}X)\otimes_R(\shift^{2n}X) \ar[r]^-{\alf{\shift^{2n}X}} \ar[d]_{\cong}^{\beta}
& (\shift^{2n}X)\otimes_R(\shift^{2n}X) \ar[d]_{\cong}^{\beta} \\
\shift^{4n}(X\otimes_R X) \ar[r]^-{\shift^{4n}\alf X} & \shift^{4n}(X\otimes_R X)
}
$$
with $\beta(x\otimes y)=x\otimes y$.
The resulting isomorphism of cokernels yields 
an isomorphism 
$$\ol\beta\colon\sw_R(\shift^{2n}X)\xra\cong\shift^{4n}\sw_R(X)
$$
given by $\ol\beta\left(\ol{x\otimes y} \right)=\ol{x\otimes y}$.
In particular, the equality $\ol\beta\left(\ol{x\otimes x} \right)=\ol{x\otimes x}$
implies that $\ol\beta$ induces an isomorphism 
$$\s_R(\shift^{2n}X)\cong\shift^{4n}\s_R(X).$$
\end{subprops}
\begin{subprops}
\label{symm04item1}
There is an exact sequence
\begin{equation} \notag
0\to
\ker(\alf X)\xra{\jinj X}
X\otimes_R X\xra{\alf X}
X\otimes_R X \xra{\ppro X} \sw_R(X)\to 0
\end{equation}
where $\jinj X$ and $\ppro X$ are the natural injection and surjection, respectively.
\end{subprops}
\begin{subprops}
\label{symm04item2}
A morphism of complexes $f\colon X\to Y$ yields a commutative
diagram
\begin{equation}
\begin{split} \notag
\xymatrix{
X\otimes_R X \ar[r]^-{\alf X} \ar[d]_{f\otimes_R f}
& X\otimes_R X \ar[d]^{f\otimes_R f} \\
Y\otimes_R Y \ar[r]^-{\alf Y} & Y\otimes_R Y.
}
\end{split}
\end{equation}
Hence, this induces a well-defined morphism on cokernels
$\sw_R(f)\colon\sw_R(X)\to\sw_R(Y)$, 
given by $\sw_R(f)\left(\ol{x\otimes y} \right)=\ol{f(x)\otimes f(y)}$.
The equality $\sw_R(f)\left(\ol{x\otimes x} \right)=\ol{f(x)\otimes f(x)}$
shows that $\sw_R(f)$ induces a
well-defined morphism $\s_R(f)\colon\s_R(X)\to\s_R(Y)$
given by $\s_R(f)\left(\ol{x\otimes y} \right)=\ol{f(x)\otimes f(y)}$. 
From the definition, one sees that the operators $\sw_R(-)$ and $\s_R(-)$ are
functorial, but Example~\ref{notadd01} shows that they are not
additive, as one might expect. 
\end{subprops}
\end{properties}

The next two results show that 
the functors $\sw_R(-)$ and $\s_R(-)$ interact well with basic
constructions.  

\begin{prop} \label{loc01}
Let $X$ be an $R$-complex.
\begin{enumerate}[\quad\rm(a)]
\item \label{loc01item1}
If $\vf\colon R\to S$ is a ring homomorphism, then there  are isomorphisms
of $S$-complexes
$\sw_S(S\otimes_R X)\cong S\otimes_R \sw_R(X)$ and 
$\s_S(S\otimes_R X)\cong S\otimes_R \s_R(X)$.
\item \label{loc01item2}
If $\p\subset R$ is a prime ideal, then there  are isomorphisms of $R_{\p}$-complexes
$\sw_{R_{\p}}(X_{\p})\cong \sw_R(X)_{\p}$ and $\s_{R_{\p}}(X_{\p})\cong \s_R(X)_{\p}$.
\end{enumerate}
\end{prop}

\begin{proof}
\eqref{loc01item1}
The vertical isomorphisms
in the following commutative diagram
are given by $\beta((s\otimes x)\otimes(t\otimes y))=(st)\otimes(x\otimes y)$:
$$\xymatrix{
(S\otimes_R X)\otimes_S (S\otimes_R X) \ar[rr]^{\alf{S\otimes_R X}} \ar[d]_{\cong}^{\beta}
&& (S\otimes_R X)\otimes_S (S\otimes_R X) \ar[d]_{\cong}^{\beta}  \\
S\otimes_R (X\otimes_R X) \ar[rr]^{S\otimes_R \alf X}
&& S\otimes_R (X\otimes_R X).
}$$
This diagram yields the first isomorphism in the next sequence.
The second isomorphism is due to the right-exactness of
$S\otimes_R-$, and the equalities are by definition.
\begin{align*}
\sw_S(S\otimes_R X)
&=\coker(\alf{S\otimes_R X})
\cong\coker(S\otimes_R \alf X)\\
&\cong S\otimes_R \coker(\alf X)
= S\otimes_R\sw_R(X).
\end{align*}
By definition, the induced isomorphism
$\ol\beta\colon \sw_S(S\otimes_R X)\xra\cong S\otimes_R\sw_R(X)$
is given by $\ol\beta\left(\ol{(s\otimes x)\otimes(t\otimes y)}\right)=\ol{(st)\otimes(x\otimes y)}$.

Let  $Y\subseteq \sw_S(S\otimes_R X)$ be the $S$-submodule
generated by elements of the form $\ol{u\otimes u}$ 
such that $u\in S\otimes_R X$ has odd degree.
That is, $Y=\ker(\opro{S\otimes X})$ where $\opro{S\otimes X}\colon\sw_S(S\otimes_R X)
\to\s_S(S\otimes_R X)$ is the natural surjection.
It is straightforward to show that $Y$ is generated over $S$ by all elements of the form
$\ol{(1\otimes x)\otimes(1\otimes x)}$. 

Let $Z\subset\sw_R(X)$ be the $R$-submodule generated 
by elements of the form $\ol{x\otimes x}$ with $x\in X$ of odd degree.
That is, we have an exact sequence of $R$-morphisms
$$0\to Z\to\sw_R(X)\xra{\opro X}\s_R(X)\to 0.$$
Tensoring with $S$ yields the next exact sequence of $S$-morphisms
$$S\otimes_RZ\to S\otimes_R\sw_R(X)\xra{S\otimes_R\opro X} S\otimes_R\s_R(X)\to 0$$
and it follows that $\ker(S\otimes_R\opro X)$ is generated over $S$ by all elements of the
form $\ol{1\otimes(x\otimes x)}$ with $x\in X$ of odd degree.
Thus, the equality
$\ol\beta\left(\ol{(1\otimes x)\otimes(1\otimes x)}\right)=\ol{1\otimes(x\otimes x)}$
shows that $\ol\beta$ induces an $S$-isomorphism $\s_S(S\otimes_R X) \cong S\otimes_R\s_R(X)$.

\eqref{loc01item2}
This follows from part~\eqref{loc01item1}
using the ring homomorphism $R\to R_{\p}$.
\end{proof}

\begin{prop} \label{s2sum01}
If $X$ and $Y$ are $R$-complexes,
then there  are isomorphisms
\begin{align}
\label{s2sum01a}
\sw_R(X\oplus Y)&\cong \sw_R(X)\oplus(X\otimes_R Y)\oplus \sw_R(Y)\\
\label{s2sum01b}
\s_R(X\oplus Y)&\cong \s_R(X)\oplus(X\otimes_R Y)\oplus \s_R(Y).
\end{align}
\end{prop}

\begin{proof}
\eqref{s2sum01a}
Tensor-distribution yields the horizontal isomorphisms
in the following commutative diagram
$$\xymatrix{
(X\oplus Y)\otimes_R (X\oplus Y) \ar[dd]_{\alf{X\oplus Y}} \ar[r]^-{\cong}
& (X\otimes_R X)\oplus (X \otimes_R Y)\oplus (Y \otimes_R X)\oplus (Y \otimes_R Y)
\ar[dd]_{\left(\begin{smallmatrix}
\alf X & 0 & 0 & 0 \\
0 & \id_{X\otimes_R Y} & -\theta_{YX} & 0 \\
0 & -\theta_{XY} & \id_{Y\otimes_R X} & 0 \\
0 & 0 & 0 & \alf Y \end{smallmatrix}\right)}
 \\ \\
(X\oplus Y)\otimes_R (X\oplus Y)
\ar[r]^-{\cong}
& (X\otimes_R X)\oplus (X \otimes_R Y)\oplus (Y \otimes_R X)\oplus (Y \otimes_R Y)
}$$
where $\theta_{UV}\colon U\otimes_R V\to V\otimes_R U$
is the tensor-commutativity isomorphism given by
$u\otimes v\mapsto (-1)^{|u||v|}v\otimes u$.
This diagram yields the first isomorphism in the following sequence
while the first equality is by definition
\begin{align*}
\sw_R(X\oplus Y)
&=\coker(\alf{X\oplus Y}) \\
&\cong\coker\left(\begin{smallmatrix}
\alf X & 0 & 0 & 0 \\
0 & \id_{X\otimes_R Y} & -\theta_{YX} & 0 \\
0 & -\theta_{XY} & \id_{Y\otimes_R X} & 0 \\
0 & 0 & 0 & \alf Y \end{smallmatrix}\right)\\
&\cong  \coker(\alf X)\oplus \coker\left(\begin{smallmatrix}
\id_{X\otimes_R Y} & -\theta_{YX}  \\
-\theta_{XY} & \id_{Y\otimes_R X}  \end{smallmatrix}\right)\oplus \coker(\alf Y) \\
& \cong \sw_R(X) \oplus(X\otimes_R Y)\oplus \sw_R(Y).
\end{align*}
The second isomorphism is by elementary linear algebra. For the
third isomorphism, using the definition of $\sw_R(-)$, we only need to
prove  $\coker(\beta) \cong X\otimes_R Y$ where
$$\beta=\left(\begin{smallmatrix}
\id_{X\otimes_R Y} & -\theta_{YX}  \\
-\theta_{XY} & \id_{Y\otimes_R X}  \end{smallmatrix}\right)\colon (X
\otimes_R Y)\oplus (Y \otimes_R X)\to (X \otimes_R Y)\oplus (Y
\otimes_R X).$$ 
We set
$$\gamma=(\id_{X\otimes_RY} \,\,\,\,  \theta_{YX})
\colon (X \otimes_R Y)\oplus (Y \otimes_R
X)\to X \otimes_R Y$$ 
which is a surjective morphism such that
$\im(\beta)\subseteq\ker(\gamma)$. Thus, there is a well-defined
surjective morphism $\ol\gamma\colon\coker(\beta)\to X\otimes_RY$
given by
$$\ol{\left(\begin{smallmatrix} x\otimes y \\ y'\otimes
x'\end{smallmatrix}\right)}\mapsto x\otimes
y+(-1)^{|x'||y'|}x'\otimes y'.$$ It remains to show that $\ol\gamma$
is injective. To this end, define $\delta\colon
X\otimes_RY\to\coker(\beta)$ by the formula
$x\otimes y\mapsto
\ol{\left(\begin{smallmatrix} x\otimes y \\
0 \end{smallmatrix}\right)}$. It is straightforward to show that
$\delta$ is a well-defined morphism and that
$\delta\ol\gamma=\id_{\coker(\beta)}$. It follows that $\ol\gamma$
is injective, hence an isomorphism, as desired.

\eqref{s2sum01b}
The isomorphism 
$\beta\colon\sw_R(X\oplus Y)\xra\cong \sw_R(X)\oplus(X\otimes_R Y)\oplus \sw_R(Y)$
from part~\eqref{s2sum01a} is given by the formula
$$\beta\left(\ol{(x,y)\otimes(x',y')}\right)
=\left(\ol{x\otimes x'},x\otimes y'+(-1)^{|x'||y|}x'\otimes y,\ol{y\otimes y'}\right).$$
Thus, for an element $(x,y)\in X\oplus Y$ of odd order
$|x|=|(x,y)|=|y|$, we have
\begin{align*}
\beta\left(\ol{(x,y)\otimes(x,y)}\right)
&=\left(\ol{x\otimes x},x\otimes y+(-1)^{|x||y|}x\otimes y,\ol{y\otimes y}\right)\\
&=\left(\ol{x\otimes x},x\otimes y-x\otimes y,\ol{y\otimes y}\right)\\
&=\left(\ol{x\otimes x},0,\ol{y\otimes y}\right).
\end{align*}
It follows that 
\begin{align*}
\s_R(X\oplus Y)\hspace{-1cm}\\
&\cong\frac{\sw_R(X)\oplus(X\otimes_R Y)\oplus \sw_R(Y)}{\left\langle\left.
\left(\ol{x\otimes x},0,\ol{y\otimes y}\right)\right|\text{$x\in X$ and $y\in Y$ have odd degree}\right\rangle}\\
&\cong\frac{\sw_R(X)}{\left\langle\left.
\ol{x\otimes x}\right|\text{$x\in X$ odd degree}\right\rangle}
\oplus
\frac{(X\otimes_R Y)}{0}
\oplus 
\frac{\sw_R(Y)}{\left\langle\left.
\ol{y\otimes y}\right|\text{$y\in Y$ odd degree}\right\rangle}
\\
&\cong \s_R(X)\oplus(X\otimes_R Y)\oplus \s_R(Y).
\end{align*}
as desired.
\end{proof}

Example~\ref{symm045} shows why we must assume that 
2 is a unit in $R$ in the next result.

\begin{prop} \label{symm05}
Assume that 2 is a unit in $R$, and let $X$ be an $R$-complex.
\begin{enumerate}[\quad\rm(a)]
\item \label{symm05item0}
The following exact sequences are split exact
\begin{gather*}
0\to\ker(\alf X)
\xra{\jinj X}X\otimes_R X
\xra{\qpro X}\im(\alf X) \to 0
\\
0\to\im(\alf X)
\xra{\iinj X}X\otimes_R X
\xra{\ppro X}\s_R(X) \to 0 
\end{gather*}
where $\iinj X$ and $\jinj X$ are the natural inclusions,
$\ppro X$ is the natural surjection,
and $\qpro X$ is induced by $\alf X$.
The splitting on the right of the first sequence is given by
$\frac{1}{2}\iinj X$, and the splitting
on the left of the second
sequence is given by $\frac{1}{2}\qpro X$.
In particular, there are isomorphisms
$$\im(\alf X)\oplus \ker(\alf X)
\cong X\otimes_R X\cong\im(\alf X)\oplus \s_R(X).$$
\item \label{symm05item1}
If $X$ is a 
bounded-below
complex of projective $R$-modules,
then so are the complexes $\im(\alf X)$, $\ker(\alf X)$ and $\s_R(X)$.
\end{enumerate}
\end{prop}

\begin{proof}
\eqref{symm05item0}
The given exact sequences come from Properties~\eqref{symm04item5}
and~\eqref{symm04item1}.
The fact that $\frac{1}{2}\alf X$ is idempotent tells us that
$\iinj X$ is a split injection with splitting given by $\frac{1}{2}\qpro X$
and $\qpro X$ is a split surjection with splitting
given by
$\frac{1}{2}\iinj X$.
The desired isomorphisms
follow immediately from the splitting of the sequences.

\eqref{symm05item1}
With the isomorphisms from part~\eqref{symm05item0},
the fact that $X\otimes_R X$ is a bounded-below
complex of projective $R$-modules
implies
that $\im(\alf X)$, $\ker(\alf X)$ and $\s_R(X)$ are also bounded-below
complexes of projective $R$-modules.
\end{proof}

The next two results  explicitly
describe  the modules in $\sw_R(X)$ and $\s_R(X)$.
Note that the difference between parts~\eqref{betti01item1}--\eqref{betti01item2}
and part~\eqref{betti01item3} shows that the behavior documented in
Example~\ref{symm045} is, in a sense, the norm, not the exception.

\begin{thm} \label{betti01}
Let $X$ be a  complex
of $R$-modules.
Fix an integer $n$ and set $h=n/2$ and
$V = \bigoplus_{m<h}  (X_m\otimes X_{n-m})$.  
\begin{enumerate}[\quad\rm(a)]
\item \label{betti01item1}
If $n$ is odd, then $\sw_R(X)_n\cong V$.
\item \label{betti01item2}
If $n\equiv 0\pmod{4}$, then 
$\sw_R(X)_n
\cong V
\bigoplus \s_R(X_h)$.
\item \label{betti01item3}
Assume that $n\equiv 2\pmod{4}$.
\begin{enumerate}[\rm(c1)]
\item \label{betti01item3a}
There is an isomorphism
\begin{align*}
\sw_R(X)_n
&\cong V
\bigoplus \frac{X_h\otimes_R X_h}{\langle
x\otimes x'+x'\otimes x
\mid x,x'\in X_h \rangle }
\end{align*}
and there is a surjection 
$\tau\colon \sw_R(X)_n\to V\oplus\wedge^2(X_h)$
with 
$\ker(\tau)$ generated by $\{\ol{x\otimes x}\in\sw_R(X)_n\mid x\in X_h\}$.
\item \label{betti01item3b}
If $X_h$ is  projective, then 
$\sw_R(X)_n
\cong V\bigoplus\wedge^2(X_h)
\bigoplus K$ for some $R$-module $K$ that is a homomorphic image of $X_h/2X_h$.
\item \label{betti01item3c}
If $X_h$ is projective and 2 is a unit in $R$, then 
$\sw_R(X)_n
\cong V\bigoplus\wedge^2(X_h)$.
\end{enumerate}
\end{enumerate}
\end{thm}

\begin{proof}
\eqref{betti01item1}
Assume that $n$ is odd.
Let $\gamma\colon(X\otimes X)_n \to V\oplus V$
be given on generators by the formula
$$\gamma(x\otimes x')=\begin{cases}
(x\otimes x',0) & \text{if $|x|<h$} \\
(0,x'\otimes x) & \text{if $|x|>h$.} 
\end{cases}$$
Since $n$ is odd, this is a well-defined  isomorphism.
Let $g\colon V\oplus V\to V\oplus V$ be given by
$g(v,v')=(v-v',v'-v)$.
This yields a commutative diagram
\begin{equation}\label{betti01item1a}
\begin{split} 
\xymatrix{
(X\otimes_R X)_n \ar[r]^{\alf X_n} \ar[d]_{\cong}^\gamma
&(X\otimes_R X)_n \ar[d]_{\cong}^\gamma \\
V\oplus V \ar[r]^g & V\oplus V.
}
\end{split}
\end{equation}
Note that the commutativity depends on the fact that $n$ is odd,
because it implies that $|x||x'|$ is even for each  $x\otimes x'\in (X\otimes_R X)_n$.

The map $f\colon V\oplus V \rightarrow V$ given by $f(v, v') =
v +v'$ is a surjective homomorphism with $\ker(f)=\langle
(v,0) - (0, v) \mid v\in V\rangle = \im(g)$.  This explains the last 
isomorphism in the next sequence
$$\sw_R(X)_n = \coker(\alf X_n)\cong\coker(g)\cong V.$$
The other isomorphism follows from diagram~\eqref{betti01item1a}.

\eqref{betti01item2}--\eqref{betti01item3}
When $n$ is even, we have a similar commutative diagram
\begin{equation}\label{betti01item2a}
\begin{split} 
\xymatrix{
(X\otimes_R X)_n \ar[r]^{\alf X_n} \ar[d]_{\cong}^{\gamma'}
&(X\otimes_R X)_n \ar[d]_{\cong}^{\gamma'} \\
V\oplus V\oplus(X_h\otimes X_h) \ar[r]^{g'} & V\oplus V\oplus(X_h\otimes X_h).
}
\end{split}
\end{equation}
where $\gamma'$ and $g'$ are  given by
\begin{align*}
\gamma'(x\otimes x')
&=\begin{cases}
(x\otimes x',0,0) & \text{if $|x|<h$} \\
(0,x'\otimes x,0) & \text{if $|x|>h$} \\
(0,0,x\otimes x') & \text{if $|x|=h$.} 
\end{cases} \\
g'(v,v',x\otimes x')
&=(v-v',v'-v,x\otimes x'-(-1)^{h^2}x'\otimes x)\\
&=(v-v',v'-v,x\otimes x'-(-1)^{h}x'\otimes x).
\end{align*}
In other words, we have $g'=g\oplus\wti\alpha$ where
$\wti\alpha\colon X_h\otimes_R X_h\to X_h\otimes_R X_h$
is given  by
$$\wti\alpha(x\otimes x'):=x\otimes x'-(-1)^{h}x'\otimes x.$$
The following sequence of isomorphisms follows directly
\begin{align*}
\sw_R(X)_n 
&= \coker(\alf X_n)
\cong\coker(g')
\cong \coker(g)\oplus\coker(\wti\alpha)\cong V\oplus\coker(\wti\alpha).
\end{align*}
If $n\equiv 0\pmod{4}$, then $h$ is even, so we have
$$\coker(\wti\alpha)\cong \frac{X_h\otimes_R X_h}{\langle
        x\otimes x'-x'\otimes x \mid x,x'\in X_h \rangle }\cong\s_R(X_h).$$
For the remainder of the proof, we assume that $n\equiv 2\pmod{4}$,
that is, that $h$ is odd. In this case, we have
\begin{equation}\label{betti01b}
\coker(\wti\alpha)
\cong \frac{X_h\otimes_R X_h}{\langle
        x\otimes x'+x'\otimes x
        \mid x,x'\in X_h \rangle }.
\end{equation}
It is straightforward to show that 
$$\langle
        x\otimes x'+x'\otimes x
        \mid x,x'\in X_h \rangle
        \subseteq \langle
        x\otimes x
        \mid x\in X_h \rangle.
$$
Hence, there is an epimorphism
$$\tau_1\colon \coker(\wti\alpha)\to \frac{X_h\otimes_R X_h}{\langle
        x\otimes x
        \mid x\in X_h \rangle }\cong\wedge^2(X_h)$$
such that
\begin{equation}\label{betti01a}
\ker(\tau_1)=\langle \ol{x\otimes x}\in\coker(\wti\alpha) \mid x\in X_h\rangle
\cong\langle \ol{x\otimes x}\in\sw_R(X)_n\mid x\in X_h\rangle.
\end{equation}
The conclusions of part~\eqref{betti01item3a} follow from setting $\tau=\id_V\oplus \tau_1$.

For the rest of the proof, we assume that $X_h$ is projective.
It follows that $\wedge^2(X_h)$ is also projective, hence
the surjection $\tau_1$ splits. Setting $K=\ker(\tau_1)$,
we have $\sw_R(X)_n\cong V\bigoplus \wedge^2(X_h)\oplus K$.
Using~\eqref{betti01b} and~\eqref{betti01a}
we see that the map
$\pi\colon X_h\to\ker(\tau_1)$ given by $x\mapsto\ol{x\otimes x}$
is surjective with $2X_h\subseteq\ker(\pi)$. It follows that $K$ is a homomorphic 
image of $X_h/2X_h$, which establishes part~\eqref{betti01item3b}.
Finally, part~\eqref{betti01item3c} follows directly from~\eqref{betti01item3b}:
if 2 is a unit in $R$, then $X_h/2X_h=0$.
\end{proof}

\begin{thm} \label{betti01'}
Let $X$ be a  complex
of $R$-modules.
Fix an integer $n$ and set $h=n/2$ and
$V = \bigoplus_{m<h}  (X_m\otimes X_{n-m})$.  
\begin{enumerate}[\quad\rm(a)]
\item \label{betti01'item1}
If $n$ is odd, then $\s_R(X)_n\cong V$.
\item \label{betti01'item2}
If $n\equiv 0\pmod{4}$, then 
$\s_R(X)_n
\cong V
\bigoplus \s_R(X_h)$.
\item \label{betti01'item3}
If $n\equiv 2\pmod{4}$, then 
$\s_R(X)_n
\cong V\bigoplus\wedge^2(X_h)$.
\end{enumerate}
\end{thm}

\begin{proof}
Set $Y=\left\langle\left.
\ol{x\otimes x}\in \sw_R(X)\right|\text{$x\in X$ odd degree}\right\rangle
\subseteq \sw_R(X)$.

\eqref{betti01'item1}--\eqref{betti01'item2}
If $n$ is odd or $n\equiv 0\pmod{4}$, then  $Y_n=0$; hence
$\s_R(X)_n\cong\sw_R(X)_n$, and the desired conclusions follow from 
Theorem~\ref{betti01}\eqref{betti01item1}--\eqref{betti01item2}.

\eqref{betti01'item3}
Assume that $n\equiv 2\pmod{4}$. The surjection 
$\tau\colon \sw_R(X)_n\to V\oplus\wedge^2(X_h)$
from Theorem~\ref{betti01}\eqref{betti01item3a}
has 
$\ker(\tau)$ generated by $\{\ol{x\otimes x}\in\sw_R(X)_n\mid x\in X_h\}$;
that is $\ker(\tau)=Y_n$, so we have
$$
V\bigoplus\wedge^2(X_h)
\cong \sw_R(X)_n/Y_n\cong
\s_R(X)_n
$$
as desired.
\end{proof}

We state the next result for $\s_R(X)$ only, because
Theorem~\ref{betti01} shows that it is only reasonable to consider such
formulas for $\sw_R(X)$ when 2 is a unit; in this case
the formulas are the same because of the isomorphism
$\sw_R(X)\cong\s_R(X)$.

\begin{cor} \label{betti011}
Let $X$ be a bounded-below complex
of finite rank free $R$-modules.
For each integer $l$, set $r_l=\rank_R(X_l)$.
Then each $R$-module $\s_R(X)_n$ is free and 
$$\rank_R((\s_R(X)_n)
=\begin{cases}
\displaystyle\sum_{m<h}r_mr_{n-m} & \text{if $n$ is odd} \\
\binom{r_h+1}{2}+ \displaystyle\sum_{m<h}r_mr_{n-m} & \text{if $n\equiv 0\pmod{4}$} \\
\binom{r_h}{2}+ \displaystyle\sum_{m<h}r_mr_{n-m} & \text{if $n\equiv 2\pmod{4}$.}
\end{cases}
$$
\end{cor}

\begin{proof}
Using the notation of Theorem~\ref{betti01'} we have
$$V = \bigoplus_{m<h}  (X_m\otimes X_{n-m})
\cong \bigoplus_{m<h}  (R^{r_m}\otimes R^{r_{n-m}})
\cong \bigoplus_{m<h}  R^{r_mr_{n-m}}
$$
and, when $n$ is even
\begin{align*}
\s_R(X_h)
&\cong \s_R(R^{r_h})\cong R^{\binom{r_h+1}{2}}
&\wedge^2(X_h)
&\cong\wedge^2(R^{r_h})\cong R^{\binom{r_h}{2}}.
\end{align*}
The desired formula now follows from Theorem~\ref{betti01'}.
\end{proof}

\begin{disc} \label{betti02}
There are several ways to present the formula in
Corollary~\ref{betti011}.  One other way to write it is the following:
$$\rank_R((\s_R(X)_n)
=\begin{cases}
\frac{1}{2}\rank_R((X\otimes_R X)_n) & \text{if $n$ is odd} \\
\frac{1}{2}\rank_R((X\otimes_R X)_n)+\frac{1}{2}r_h & \text{if $n\equiv 0\pmod{4}$} \\
\frac{1}{2}\rank_R((X\otimes_R X)_n)-\frac{1}{2}r_h & \text{if $n\equiv 2\pmod{4}$.}
\end{cases}
$$
Another way is in terms of generating functions: For a complex $Y$ of free $R$-modules,
set $P^R_Y(t)=\sum_n\rank_R(Y_n)t^n$.  (Note that this is not usually the same
as the Poincar\'{e} series of $Y$.  It is the same if and only if 
$R$ is local and $Y$ is minimal.)
Using the previous display, we can then write
\begin{equation} \label{betti02eq1} 
P^R_{\s_R(X)}(t)=\textstyle\frac{1}{2}\left[
P^R_{X\otimes_R X}(t)+P^R_X(-t^2)\right]=\frac{1}{2}\left[
P^R_{X}(t)^2+P^R_X(-t^2)\right].
\end{equation}
We make use of this expression several times in what follows.
\end{disc}

\section{Homological Properties of $\s_R(X)$} \label{sec03}

This section
documents the homological and homotopical aspects of the functor $\s_R(-)$.
It also contains our proof of Theorem~\ref{intthmC}
from the introduction. 
We assume throughout this section that $2$ is a unit in $R$,
and it follows that $\s_R(X)\cong\sw_R(X)$ 
via the natural map for all $X$.

We begin with
the following result showing that $\s_R(X)$ exhibits properties
similar to those for $X\otimes_R X$ noted in Fact~\ref{fact01}.
Example~\ref{symm045} shows what goes wrong in 
part~\eqref{symm09item2} when
$\inf(X)$ is odd: 
we have $\s_R(\shift R)= 0$
so $\inf(\s_R(\shift R))=\infty>2=2\inf(\shift R)$.
Note that we do not need $R$ to be local in either 
part of this result.

\begin{prop} \label{symm09}
Assume that 2 is a unit in $R$ and let $X$ be a bounded-below complex
of projective $R$-modules.
\begin{enumerate}[\quad\rm(a)]
\item \label{symm09item1}
There is an inequality $\inf(\s_R(X))\geq 2\inf(X)$
and there is  an isomorphism 
$$\HH_{2\inf(X)}(\s_R(X))\cong \begin{cases} \s_R(\HH_{\inf(X)}(X)) & \text{if  $\inf(X)$  is even,}\\
\displaystyle\frac{\HH_{\inf(X)}(X)\otimes \HH_{\inf(X)}(X)}{\langle x\otimes y+ y\otimes x \mid
    x,y \in \HH_{\inf(X)}(X)\rangle} & \text{if  $\inf(X)$ is odd.}
\end{cases}$$
\item \label{symm09item2}
Assume that $R$ is noetherian and that 
$\HH_{\inf(X)}(X)$ is finitely generated.
If $\inf(X)$ is even, then $\inf(\s_R(X))= 2\inf(X)$.
\end{enumerate}
\end{prop}

\begin{proof}
\eqref{symm09item1}
Set $i=\inf(X)$. Proposition~\ref{symm05}\eqref{symm05item1} yields an
isomorphism 
$$\im(\alf X)\oplus \s_R(X) \cong X\otimes_R X.$$
This isomorphism yields the first
inequality in the next sequence
$$\inf(\s_R(X))\geq \inf(X\otimes_R X)
\geq 2i$$ while the second inequality is in Fact~\ref{fact01}. 

The split exact sequences from Proposition~\ref{symm05}\eqref{symm05item0}
fit together in the following commutative diagram
\begin{equation}
\begin{split} \label{seqdiag01}
\xymatrix{
0 \ar[r]
& \ker(\alf X)\ar[r]^-{\jinj X}
& X\otimes_R X \ar[r]^{\qpro X} \ar[d]_{\qpro X} \ar[rd]^{\alf X}
& \im(\alf X) \ar[r] \ar[d]^{\iinj X}
& 0 \\
& 0 \ar[r]
& \im(\alf X) \ar[r]^-{\iinj X}
& X\otimes_R X \ar[r]^{\ppro X}
& \s_R(X) \ar[r]
& 0.
}
\end{split}
\end{equation}
Define $\wti\alpha\colon \HH_{i}(X)\otimes_R \HH_{i}(X)\to \HH_{i}(X)\otimes_R \HH_{i}(X)$
by the formula 
$$\ol{x}\otimes\ol{x'}\mapsto\ol{x}\otimes\ol{x'}
-(-1)^{i^2}\ol{x'}\otimes\ol{x}=\ol{x}\otimes\ol{x'}
-(-1)^{i}\ol{x'}\otimes\ol{x}.$$
It is straightforward to  show that the following diagram commutes
\begin{equation}
\begin{split} \label{seqdiag01'}
\xymatrix{
\HH_{2i}(X\otimes_R X) \ar[r]^{\HH_{2i}(\alf X)} \ar[d]_{\cong}^{\gamma}
& \HH_{2i}(X\otimes_R X) \ar[d]^{\cong}_{\gamma}\\
\HH_{i}(X)\otimes_R \HH_{i}(X) \ar[r]^-{\wti\alpha}
& \HH_{i}(X)\otimes_R \HH_{i}(X).
}
\end{split}
\end{equation}
where the isomorphism $\gamma$ is
from Fact~\ref{fact01}. Together,
diagrams~\eqref{seqdiag01} and~\eqref{seqdiag01'} yield the next commutative diagram
\begin{equation}
\begin{split} \notag
\xymatrix{
\HH_i(X)\otimes \HH_i(X) \ar[rr]^{\HH_{2i}(\qpro X)\gamma^{-1}} \ar[d]^{\HH_{2i}(\qpro X)\gamma^{-1}}
\ar[rrd]^{\wti\alpha}
&& \HH_{2i}(\im(\alf X)) \ar[rr] \ar[d]^{\gamma\HH_{2i}(\iinj X)}
&& 0 \\
\HH_{2i}(\im(\alf X)) \ar[rr]^-{\gamma\HH_{2i}(\iinj X)}
&& \HH_{i}(X)\otimes \HH_{i}(X) \ar[rr]^-{\HH_{2i}(\ppro X)\gamma^{-1}}
&& \HH_{2i}(\s_R(X)) \ar[r]
& 0.
}
\end{split}
\end{equation}
whose rows are exact because the rows of diagram~\eqref{seqdiag01}
are split exact. A straightforward diagram-chase yields the equality
$\ker(\HH_{2i}(\ppro X)\gamma^{-1})=\im(\wti\alpha)$
and so
$$\HH_{2i}(\s_R(X))\cong \frac{\HH_{i}(X)\otimes_R
\HH_{i}(X)}{\im(\wti\alpha)}\cong\begin{cases}
  \s_R(\HH_i(X))& \text{if $i$  is even} \\ \displaystyle\frac{H_i(X)\otimes
    H_i(X)}{\langle x\otimes y + y\otimes x \mid x,y\in
    \HH_i(X)\rangle} & \text{ if $i$ is odd.}\end{cases}$$

\eqref{symm09item2}
Using part~\eqref{symm09item1}, it suffices to to show that $\s_R(\HH_i(X))\neq
0$ where $i=\inf(X)$. Fix a maximal ideal  $\m\in\supp_R(\HH_i(X))$, and set $k=R/\m$.  
Using the isomorphisms
$$k\otimes_R\HH_i(X)
\cong (k\otimes_{R_{\m}}R_{\m})\otimes_R\HH_i(X)
\cong k\otimes_{R_{\m}}\HH_i(X)_{\m}
\cong k\otimes_{R_{\m}}\HH_i(X_{\m})$$
Nakayama's Lemma implies that $k\otimes_R\HH_i(X)$
is a nonzero $k$-vector space of finite rank, say $k\otimes_R\HH_i(X)\cong k^r$.
In the following sequence, the first and third isomorphisms are well-known; see, e.g.,
\cite[(A2.2.b) and (A2.3.c)]{eisenbud:ca}:
$$
k\otimes_R\s_R(\HH_i(X))
\cong\s_k(k\otimes_R\HH_i(X))
\cong \s_k(k^r)
\cong k^{\binom{r+1}{r-1}}\neq 0.$$
It follows that $\s_R(\HH_i(X))\neq
0$, as desired.
\end{proof}

The next result establishes the homotopy-theoretic properties of the functor $\s_R(-)$.
Example~\ref{symm05c} shows that conclusion fails when
2 is not a unit in $R$.
Note that we cannot reduce
part~\eqref{symm05aitem1} to the case $g=0$ by replacing $f$ by
$f-g$, as Example~\ref{notadd01} shows that $\s_R(f-g)$ might not
equal $\s_R(f)-\s_R(g)$.

\begin{thm} \label{symm05a}
Assume that 2 is a unit in $R$, and let $X$ and $Y$ be $R$-complexes.
Fix morphisms $f,g\colon X\to Y$ and $h\colon Y\to X$.
\begin{enumerate}[\quad\rm(a)]
\item \label{symm05aitem1}
If $f$ and $g$ are homotopic, then $\s_R(f)$ and $\s_R(g)$ are homotopic.
\item \label{symm05aitem2}
If $f$ is a homotopy equivalence with homotopy inverse $h$,
then $\s_R(f)$ is a homotopy equivalence with homotopy inverse $\s_R(h)$.
\end{enumerate}
\end{thm}

\begin{proof}
\eqref{symm05aitem1} Fix a homotopy $s$ from $f$ to $g$ as in
Definition~\ref{notn02'}. Define 
\begin{align*}
f\otimes_R s+s\otimes_R g&=\{(f\otimes_R s+s\otimes_R g)_n\colon (X\otimes_RX)_n\to
(Y\otimes_R Y)_{n+1}\}\\
g\otimes_R s+s\otimes_R f&=\{(g\otimes_R s+s\otimes_R f)_n\colon (X\otimes_RX)_n\to
(Y\otimes_R Y)_{n+1}\}
\end{align*}
on each generator $x\otimes x'\in(X\otimes_RX)_n$ by the formulas
\begin{align*}
(f\otimes_R s+s\otimes_R g)_n(x\otimes x')
&:=(-1)^{|x|}f_{|x|}(x)\otimes s_{|x'|}(x')+s_{|x|}(x)\otimes g_{|x'|}(x')\\
(g\otimes_R s+s\otimes_R f)_n(x\otimes x')
&:=(-1)^{|x|}g_{|x|}(x)\otimes s_{|x'|}(x')+s_{|x|}(x)\otimes f_{|x'|}(x').
\end{align*}
One checks readily that the sequences
$f\otimes_R s+s\otimes_R g$ and $g\otimes_R s+s\otimes_R f$ are
homotopies from $f\otimes_R f$ to $g\otimes_R g$. As 2 is a
unit in $R$, it follows that the sequence
$$\sigma=\textstyle\frac{1}{2}(f\otimes_R s+s\otimes_R g+g\otimes_R s+s\otimes_R f)
$$
is also a homotopy
from $f\otimes_R f$ to $g\otimes_R g$.
It is straightforward to show that
$\sigma_n\alf X_n=\alf Y_{n+1}\sigma_n $
for all $n$.
Using  the fact that $\sigma$ is a homotopy
from $f\otimes_R f$ to $g\otimes_R g$, it is thus straightforward to show that $\sigma$ induces
a homotopy $\ol\sigma$
from $\s_R(f)$ to $\s_R(g)$ by the formula 
$\ol\sigma_n\left(\ol{x\otimes x'}\right)=\ol{\sigma_n(x\otimes x')}$.

\eqref{symm05aitem2}
By hypothesis, the composition $hf$ is homotopic to $\id_X$.
Part~\eqref{symm05aitem1}
implies that $\s_R(hf)=\s_R(h)\s_R(f)$ is homotopic to
$\s_R(\id_X)=\id_{\s_R(X)}$.
The same logic implies that $\s_R(f)\s_R(h)$ is homotopic to
$\id_{\s_R(Y)}$, and hence the desired conclusions.
\end{proof}

For the next results,
Examples~\ref{koszul01''} and~\ref{symm05c} show why we need to assume that
$X$ and $Y$ are bounded-below
complexes of projective $R$-modules and 2 is a unit in $R$.

\begin{cor} \label{symm05b}
Assume that 2 is a unit in $R$, and let $X$ and $Y$ be bounded-below
complexes of projective $R$-modules.
\begin{enumerate}[\quad\rm(a)]
\item \label{symm05bitem2}
If $f\colon X\to Y$ is a quasiisomorphism, then
so is $\s_R(f)\colon\s_R(X)\to \s_R(Y)$.
\item \label{symm05bitem3}
If $X\simeq Y$, then $\s_R(X)\simeq \s_R(Y)$.
\end{enumerate}
\end{cor}

\begin{proof}
\eqref{symm05bitem2}
Our assumptions imply that $f$ is a homotopy equivalence by Fact~\ref{fact0001}, so the desired
conclusion follows from Theorem~\ref{symm05a}\eqref{symm05aitem2}.

\eqref{symm05bitem3}
Assume $X\simeq Y$.  Because $X$ and $Y$ are bounded-below
complexes of projective $R$-modules,
there is a quasiisomorphism
$f\colon X\xra{\simeq} Y$.
Now apply part~\eqref{symm05bitem2}.
\end{proof}

\begin{cor} \label{loc01'}
If 2 is a unit in $R$ and $X$ is a bounded-below complex of
projective $R$-modules, then
there is a containment $\supp_R(\s_R(X))\subseteq\supp_R(X)$.
\end{cor}

\begin{proof}
Fix a prime ideal $\p\not\in\supp_R(X)$.  It suffices
to show $\p\not\in\supp_R(\s_R(X))$.
The first isomorphism in the following sequence is from 
Proposition~\ref{loc01}\eqref{loc01item2}
$$\s_R(X)_{\p}
\cong \s_{R_{\p}}(X_{\p}) \simeq \s_{R_{\p}}(0) = 0.$$ The
quasiisomorphism follows from
Corollary~\ref{symm05b}\eqref{symm05bitem3} because $X_{\p}\simeq
0$.
\end{proof}

The following result is key for our proof of Theorem~\ref{intthmC}.

\begin{thm} \label{symm07}
Assume that $R$ is noetherian and local and that 2  is a unit in $R$.
Let $X$ be a bounded-below complex
of finite-rank free $R$-modules.
The following conditions are equivalent:
\begin{enumerate}[\quad\rm(i)]
\item \label{symm07item1a}
the surjection $\ppro X\colon X\otimes_R X\to \s_R(X)$ is a quasiisomorphism;
\item \label{symm07item1c}
$\im(\alf X)\simeq 0$;
\item \label{symm07item1b}
the injection $\jinj X\colon \ker(\alf X)\to X\otimes_R X$ is a quasiisomorphism;
\item \label{symm07item1d}
either $X\simeq 0$ or $X\simeq\shift^{2n} R$ for some integer $n$.
\end{enumerate}
\end{thm}

\begin{proof}\eqref{symm07item1a}
The biimplications
\eqref{symm07item1a}$\iff$\eqref{symm07item1c}$\iff$\eqref{symm07item1b} 
follow easily from the long exact sequences associated to the exact
sequences in Proposition~\ref{symm05}\eqref{symm05item0}.

\eqref{symm07item1d}$\implies$\eqref{symm07item1a}.
If $X\simeq 0$, then $X\otimes_R X\simeq 0\simeq\s_R(X)$ and so
$\ppro X$ is trivially a quasiisomorphism; see Fact~\ref{fact01}
and Example~\ref{symm04item4}.

Assuming that $X\simeq \shift^{2n} R$,  there is a quasiisomorphism
$\gamma\colon R\res\shift^{-2n}X$.
The commutative diagrams from~\eqref{symm04item3} 
and~\eqref{symm04item2} can be combined and augmented to form the following
commutative diagram:
$$\xymatrix{
R\otimes_R R \ar[r]^-{\alf R} \ar[d]_{\simeq}^{\gamma\otimes\gamma}
& R\otimes_R R \ar[r]^-{\ppro R}_-{\cong} \ar[d]_{\simeq}^{\gamma\otimes\gamma}
& \s_R(R) \ar[r] \ar[d]_{\simeq}^{\s(\gamma)}
& 0 \\
(\shift^{-2n}X)\otimes_R(\shift^{-2n}X) \ar[r]^-{\alf{\shift^{-2n}X}} \ar[d]_{\cong}
& (\shift^{-2n}X)\otimes_R(\shift^{-2n}X) \ar[d]_{\cong} \ar[r]^-{\ppro{\shift^{-2n}X}}
&\s_R(\shift^{-2n}X) \ar[r] \ar[d]_{\cong}
& 0 \\
\shift^{-4n}(X\otimes_R X) \ar[r]^-{\shift^{-4n}\alf X} 
& \shift^{-4n}(X\otimes_R X) \ar[r]^-{\shift^{-4n}\ppro X} 
& \shift^{-4n}\s_R(X) \ar[r]
& 0.
}
$$
The morphism $\gamma\otimes \gamma$ is a quasiisomorphism by Fact~\ref{fact01},
and $\s(\gamma)$ is a quasiisomorphism by Corollary~\ref{symm05b}\eqref{symm05bitem2}.
One checks readily that $\alf R=0$ so $\ppro R$ is an isomorphism.
The diagram shows that $\ppro{\shift^{-2n}X}$ is a quasiisomorphism,
and hence so is $\shift^{-4n}\ppro X$. It follows that $\ppro X$ is a quasiisomorphism,
as desired.

\eqref{symm07item1a}$\implies$\eqref{symm07item1d}.
Assume that the surjection $\ppro X\colon X\otimes_R X\to \s_R(X)$ is a
quasiisomorphism and $X\not\simeq 0$. 

Case 1: $X$ is minimal. 
This implies that $X\otimes_RX$ is minimal. Also, since $\s_R(X)$ is a 
direct summand of $X\otimes_RX$, it follows that $\s_R(X)$ is also minimal.
The fact that
$\ppro X$ is a quasiisomorphism then implies that it is an isomorphism;
see Fact~\ref{fact0001''}.
This explains the  second equality in the next sequence
$$P^R_{X}(t)^2
=P^R_{X\otimes_R X}(t)
=P^R_{\s_R(X)}(t) 
=\textstyle\frac{1}{2}\left[P^R_{X}(t)^2+P^R_X(-t^2)\right].$$ 
The third equality is from equation~\eqref{betti02eq1}.
It follows that 
\begin{equation} \label{poinc02a}
P^R_{X}(t)^2=P^R_X(-t^2).
\end{equation}
Let 
$i = \inf(X)$ and note that $r_i\geq 1$.  
Set $r_n=\rank_R(X_{n-i})$ for each $n$ and $Q(t)=\sum_{n=0}^{\infty}r_{n+i}t^n$,
so that we have $P^R_X(t)=t^iQ(t)$. Equation~\eqref{poinc02a} then reads as
$t^{2i}Q(t)^2=(-1)^it^{2i}Q(-t^2)$, that is, we have
\begin{equation} \label{poinc02b}
Q(t)^2-(-1)^iQ(-t^2)=0.
\end{equation}
If $i$ were odd, then this would say $Q(t)^2+Q(-t^2)=0$, contradicting
Lemma~\ref{poinc01}\eqref{poinc01item1}.
It follows that $i= 2n$ for some
$n$. Equation~\eqref{poinc02b} then says $Q(t)^2-Q(-t^2)=0$, and so
Lemma~\ref{poinc01}\eqref{poinc01item2} implies that $Q(t)=1$.
This says that $P^R_X(t)=t^i=t^{2n}$ and so $X\cong \shift^{2n}R$,
as desired.

Case 2: the general case.
Let $\delta\colon P\res X$ be a minimal free resolution.
We again augment the commutative diagram from~\eqref{symm04item2} 
$$\xymatrix{
P\otimes_R P \ar[r]^-{\alf P} \ar[d]_{\simeq}^{\delta\otimes\delta}
& P\otimes_R P \ar[r]^-{\ppro P} \ar[d]_{\simeq}^{\delta\otimes\delta}
& \s_R(P) \ar[r] \ar[d]_{\simeq}^{\s(\delta)}
& 0 \\
X\otimes_RX \ar[r]^-{\alf{X}} 
& X\otimes_RX \ar[r]^-{\ppro{X}}_-{\simeq}
&\s_R(X) \ar[r] 
& 0.
}
$$
This implies that $\ppro P$ is a quasiisomorphism. Since $P$ is minimal,
Case 1 implies that either 
$P\simeq 0$ or $P\simeq\shift^{2n} R$ for some integer $n$.
Since we have $X\simeq P$, the desired conclusion follows.
\end{proof}

\begin{disc} \label{disc0002}
One can remove the local assumption and change the word ``free'' to ``projective'' in
Theorem~\ref{symm07} if one replaces condition~\eqref{symm07item1d} with the following
condition:
(iv') for every maximal ideal $\m\subset R$, one has 
either $X_{\m}\simeq 0$ or $X_{\m}\simeq\shift^{2n} R_{\m}$
for some integer $n$. (Here the integer $n$ depends on the choice of $\m$.)
While this gives the illusion of greater generality, this version is equivalent to
Theorem~\ref{symm07} because each of the conditions
\eqref{symm07item1a}--\eqref{symm07item1b} and (iv') is local.
Hence, we state only the local versions of our results, with the knowledge
that nonlocal versions are direct consequences.
On the other hand, Example~\ref{ex0002} shows that one needs to take care when
removing the local hypotheses from our results. 
\end{disc}

We next show how Theorem~\ref{intthmC} is a consequence of 
Theorem~\ref{symm07}.

\begin{para} \label{thmCpf}
\emph{Proof of Theorem~\ref{intthmC}.}
The assumption $X_{\p}\simeq S_{\p}\neq 0$ for each $\p\in\ass(R)$ implies
$X\not\simeq 0$ and $\inf(X)\leq\inf(X_{\p})=0$.
On the other hand, since $X_n=0$ for all $n<0$, we know
$\inf(X)\geq 0$,  so $\inf(X)=0$.  

Consider the split exact
sequence from Proposition~\ref{symm05}\eqref{symm05item0}
\begin{equation} \label{thmBpfeq02} 
0\to\im(\alf X)
\xra{\iinj X}X\otimes_S X
\xra{\ppro X}\s_S(X) \to 0.
\end{equation}
This sequence splits,  so
$\HH_n(\im(\alf X))\hookrightarrow\HH_n(X\otimes_SX)$ for each $n$;
hence
\begin{equation} \label{thmBpfeq01} 
\ass_R(\HH_n(\im(\alf X)))\subseteq\ass_R(\HH_n(X\otimes_SX))
\subseteq\ass(R).
\end{equation}
For each $\p\in\ass(R)$ localization of~\eqref{thmBpfeq02} yields the exactness
of the rows of the following commutative diagram; see also Proposition~\ref{loc01}\eqref{loc01item2}.
$$\xymatrix{
0\ar[r]
&\im(\alf X)_{\p} \ar[r]^-{(\iinj X)_{\p}} \ar[d]_{\cong}
&(X\otimes_S X)_{\p} \ar[r]^-{(\ppro X)_{\p}}\ar[d]_{\cong}
&\s_S(X)_{\p} \ar[r] \ar[d]_{\cong}
& 0 \\
0\ar[r]
&\im(\alf{X_{\p}}) \ar[r]^{\iinj{X_{\p}}}
&X_{\p}\otimes_{S_{\p}} X_{\p}\ar[r]^-{\ppro{X_{\p}}}
&\s_{S_{\p}}(X_{\p}) \ar[r]
& 0
}$$
The quasiisomorphism $X_{\p}\simeq S_{\p}$ implies
that $\ppro{X_{\p}}$
is also a quasiisomorphism by Theorem~\ref{symm07}, and so
the previous sequence implies
$\im(\alf X)_{\p}\cong \im(\alf{X_{\p}})\simeq 0$
for each $\p\in\ass(R)$.
For each $n$ and $\p$, this implies
$\HH_n(\im(\alf X))_{\p}
\cong \HH_n(\im(\alf X)_{\p})=0$;
the containment in~\eqref{thmBpfeq01}
implies $\HH_n(\im(\alf X))=0$ for each $n$, that is $\im(\alf X)\simeq 0$.
Hence, Theorem~\ref{symm07} implies
$X\simeq S$.  \qed
\end{para}

The next result is a companion to Theorem~\ref{symm07}.

\begin{thm} \label{symm07''}
Assume that $R$ is noetherian and local, and that 2  is a unit in $R$.
Let $X$ be a bounded-below complex
of finite rank free $R$-modules.
The following conditions are equivalent:
\begin{enumerate}[\quad\rm(i)]
\item \label{symm07item2a}
the morphism $\alf X\colon X\otimes_R X\to X\otimes_R X$ is a quasiisomorphism;
\item \label{symm07item2b}
the surjection $\qpro X\colon X\otimes_R X\to \im(\alf X)$ is a quasiisomorphism;
\item \label{symm07item2c}
the injection $\iinj X\colon \im(\alf X)\to X\otimes_R X$ is a quasiisomorphism;
\item \label{symm07item2d}
$\s_R(X)\simeq 0$;
\item \label{symm07item2e}
$\ker(\alf X)\simeq 0$;
\item \label{symm07item2f}
$X\simeq 0$ or $X\simeq \shift^{2n+1}R$ for some integer $n$. 
\end{enumerate}
\end{thm}

\begin{proof}
The biimplications
\eqref{symm07item2b}$\iff$\eqref{symm07item2e} and
\eqref{symm07item2c}$\iff$\eqref{symm07item2d}
follow easily from the long exact sequences associated to the exact sequences in
Proposition~\ref{symm05}\eqref{symm05item0}.

For the remainder of the proof, we use the
easily verified fact that 
the exact sequences from Proposition~\ref{symm05}\eqref{symm05item0}
fit together in the following commutative diagram
\begin{equation}
\begin{split}
\label{seqdiag02}
\xymatrix{
0 \ar[r]
& \ker(\alf X)\ar[r]^-{\jinj X}
& X\otimes_R X \ar[r]^{\qpro X} \ar[d]_{\qpro X} \ar[rd]^{\alf X}
& \im(\alf X) \ar[r] \ar[d]^{\iinj X}
& 0 \\
& 0 \ar[r]
& \im(\alf X) \ar[r]^-{\iinj X}
& X\otimes_R X \ar[r]^{\ppro X}
& \s_R(X) \ar[r]
& 0
}
\end{split}
\end{equation}
and we recall that these exact sequences split.

\eqref{symm07item2a}$\implies$\eqref{symm07item2d}. 
Assume that $\alf X$ is a quasiisomorphism.

Case 1: $X$ is minimal.
Since $X$ is
minimal, the same is true of $X\otimes_R X$, so the fact that
$\alf X$ is a quasiisomorphism 
implies that $\alf X$ is an isomorphism; see Fact~\ref{fact0001''}. Hence, we have
$\s_R(X)=\coker(\alf X)= 0$. 

Case 2: the general case. Let $f\colon P\res X$ be a minimal free resolution.
The commutative diagram from~\eqref{symm04item2}
$$\xymatrix{
P\otimes_R P \ar[r]^-{\alf P} \ar[d]_{f\otimes_R f}^{\simeq}
& P\otimes_R P \ar[d]_{f\otimes_R f}^{\simeq} \\
X\otimes_R X \ar[r]^-{\alf X}_{\simeq} & X\otimes_R X
}$$
shows that $\alf P$ is a quasiisomorphism; see Fact~\ref{fact01}.
Using Corollary~\ref{symm05b}\eqref{symm05bitem2}, Case 1 implies that
$\s_R(X)\simeq\s_R(P)=0$. 

\eqref{symm07item2d}$\implies$\eqref{symm07item2e}
and \eqref{symm07item2d}$\implies$\eqref{symm07item2a}
and \eqref{symm07item2d}$\implies$\eqref{symm07item2f}. 
Assume that $\s_R(X)\simeq 0$.

Case 1: $X$ is minimal.
In this case $X\otimes_R X$ is also minimal.
The bottom row of~\eqref{seqdiag02} is split exact, so
this implies that $\s_R(X)$
is also minimal.  Hence, the condition $\s_R(X)\simeq 0$ implies that
$\s_R(X)= 0$. 
Hence, the following sequence is split exact
$$0\to\Ker(\alf X)\xra{\jinj X}X\otimes_R X\xra{\alf X}X\otimes_R X\to 0.$$
Since each $R$-module $\ker(\alf X)_n$ is free of finite rank,
the additivity of rank implies that $\ker(\alf X)_n=0$ for all $n$, that is
$\ker(\alf X)=0$. The displayed sequence then shows that $\alf X$ is an isomorphism.

Assume for the rest of this case that $X\not\simeq 0$ and set
$i=\inf (X)$. If $i$ is even,
then Proposition~\ref{symm09} implies that $\infty = \inf(\s_R(X)) =
2i<\infty$, a contradiction.  Thus $i$ is
odd. 
As before, there is a formal power series $Q(t)=\sum_{i=0}^{\infty}r_it^i$ with nonnegative integer  
coefficients such that $r_0\neq 0$ and 
$P^R_X(t)=t^iQ(t)$.
Since $\s_R(X) = 0$ the following 
formal equalities are from~\eqref{betti02eq1}:
$$0
=P^R_{\s_R(X)}(t)=\textstyle\frac{1}{2}\left[P^R_{X}(t)^2+P^R_X(-t^2)\right]
=\frac{1}{2}\left[t^{2i}Q(t)^2-t^{2i}Q(-t^2)\right].$$
It follows that $Q(t)^2-Q(-t^2)=0$, so Lemma~\ref{poinc01}\eqref{poinc01item2}
implies that $Q(t)=1$.
This
implies that $P^R_X(t) = t^i$ and so $X\cong \shift^i R$.  

Case 2: the general case. Let $f\colon P\to X$ be a minimal free resolution.
Corollary~\ref{symm05b} implies that $\s_R(P)\simeq\s_R(X)\simeq 0$,
so Case 1 also implies that either $X\simeq P\simeq 0$ or
$X\simeq P\simeq\shift^{2n+1}R$ for some integer $n$.
Case 1 also implies that $\ker(\alf P)=0$ and $\alf P$ is an isomorphism.
The commutative diagram from~\eqref{symm04item2}
$$\xymatrix{
P\otimes_R P \ar[r]^-{\alf P}_-{\cong} \ar[d]_{f\otimes_R f}^{\simeq}
& P\otimes_R P \ar[d]_{f\otimes_R f}^{\simeq} \\
X\otimes_R X \ar[r]^-{\alf X} & X\otimes_R X
}$$
shows that $\alf X$ is a quasiisomorphism; see Fact~\ref{fact01}.
Since $\s_R(X)\simeq 0$, the bottom row of~\eqref{seqdiag02}
shows that $\iinj X$ is a quasiisomorphism. Since $\alf X$ is also
a quasiisomorphism, the commutativity of~\eqref{seqdiag02} shows that
$\qpro X$ is a quasiisomorphism as well. Hence, the 
top row of~\eqref{seqdiag02}
implies that $\ker(\alf X)\simeq 0$.

\eqref{symm07item2e}$\implies$\eqref{symm07item2d}. 
Argue as in the proof of the implication
\eqref{symm07item2d}$\implies$\eqref{symm07item2e}. 

\eqref{symm07item2f}$\implies$\eqref{symm07item2d}.
If $X\simeq 0$, then $\s_R(X)\simeq\s_R(0)=0$ by
Example~\ref{symm04item4}
and Corollary~\ref{symm05b}\eqref{symm05bitem3}.
If $X\simeq\shift^{2n+1}R$ for some integer $n$, then
Corollary~\ref{symm05b}\eqref{symm05bitem3}
explains the first quasiisomorphism in the next sequence
$$\s_R(X)
\simeq\s_R(\shift^{2n+1}R)
\simeq\s_R(\shift^{2n}(\shift R))
\simeq\shift^{4n}\s_R(\shift R)
\simeq 0.$$
The second quasiisomorphism is because 
of the isomorphism $\shift^{2n+1}R\cong \shift^{2n}(\shift R)$;
the third quasiisomorphism is from~\eqref{symm04item3};
and the last quasiisomorphism follows from Example~\ref{symm045}.
\end{proof}

\begin{cor} \label{s2fpd01}
Assume that $R$ is noetherian and local, and that 2 is a unit in $R$.
Let $X$ be a bounded-below complex of
finite rank free $R$-modules.
Then $\s_R(X)$ has finite projective dimension if and only if
$X$ has finite projective dimension.
\end{cor}

\begin{proof}
Assume first that $\pd_R(X)$ is finite,
and let $P\res X$ be a bounded free resolution.
It follows that $P\otimes_R P$ is a bounded complex of free $R$-modules.
Hence, the isomorphism $P\otimes_R P\cong \s_R(P)\oplus\im(\alf P)$
from Proposition~\ref{symm05}\eqref{symm05item1} implies that
$\s_R(P)$ is a bounded complex of free $R$-modules. 
The quasiisomorphism $\s_R(X)\simeq\s_R(P)$
from Corollary~\ref{symm05b}\eqref{symm05bitem3} implies that
$\s_R(X)$ has finite projective dimension.

For the converse, assume that $X$ has infinite projective dimension.
Let $P\res X$ be a minimal free resolution, which is necessarily unbounded.
As we have noted previously, the fact that $P$ is minimal implies
that $\s_R(P)\res\s_R(X)$ is a minimal free resolution,
so it suffices to show that $\s_R(P)$ is unbounded;
see Fact~\ref{fact0001''}.

Set
$r_n=\rank_R(P_n)$ for each integer $n$.
Since $P$ is unbounded, we know  that, for each
integer $n$, there exist integers $p$ and $q$ such that
$q>p>n$ and such that the free $R$-modules $P_p$ and $P_q$
are nonzero, that is, such that $r_pr_q\neq 0$.
The inequality $q>p$ implies $p<(p+q)/2$.
For each $n\geq 0$, we then have $p+q>2n$
and 
$$\rank_R(\s_R(P)_{p+q})
\geq\sum_{m<(p+q)/2}r_mr_{p+q-m}
\geq r_pr_q>0.$$
The first inequality is from
Corollary~\ref{betti011};
the second inequality follows from the inequality $p<(p+q)/2$;
and the third inequality follows from the assumption $r_pr_q\neq 0$.
This shows that for each $n\geq 0$,
that is an integer $m=p+q>n$ such that $\s_R(P)_m\neq 0$. 
This means that $\s_R(P)$ is unbounded, as desired.
\end{proof}

The final result of this section is a refinement of the previous result.
It characterizes the complexes $X$ such that $\s_R(X)\simeq
\shift^j R$ for some integer $j$.

\begin{cor} \label{s2fpd02}
Assume that $R$ is noetherian and local, and that 2 is a unit in $R$.
Let $X$ be a bounded-below complex of
finite rank free $R$-modules.
The folowing conditions are equivalent:
\begin{enumerate}[\quad\rm(i)]
\item \label{s2fpd02c}
$X\simeq \shift^{2n} R$ for some $n$ or $X \simeq (\shift^{2n+1}
R)\oplus(\shift^{2m+1}R)$ for some $n$ and $m$; 
\item \label{s2fpd02a}
$\s_R(X)\simeq
\shift^j R$ for some even integer $j$;
\item \label{s2fpd02b}
$\s_R(X)\simeq
\shift^j R$ for some integer $j$.
\end{enumerate}
\end{cor}

\begin{proof}
\eqref{s2fpd02c}$\implies$\eqref{s2fpd02a}.
If $X\simeq \shift^{2n}R$, 
then we have
$$\s_R(X) \simeq \s_R(\shift^{2n}R) \cong
\shift^{4n}\s_R(R) \cong\shift^{4n}R$$
by~\eqref{symm04item3}, Example~\ref{symm04item4}  and
Corollary~\ref{symm05b}\eqref{symm05bitem3}.
If $X\simeq (\shift^{2n+1}
R)\oplus(\shift^{2m+1}R)$, then Proposition~\ref{s2sum01} implies
$$\s_R(X) \simeq \s_R(\shift^{2n+1} R)\oplus \left[(\shift^{2n+1} R)\otimes_R
(\shift^{2m+1} R)\right]\oplus \s_R(\shift^{2m+1} R).$$  Example~\ref{symm045}
implies that the first and last summands on the right side are 0, so   
$$\s_R(X)
\cong \shift^{2n+1} R\otimes_R \shift^{2m+1} R\cong \shift^{2n+2m+2}
R.$$

\eqref{s2fpd02a}$\implies$\eqref{s2fpd02b}.
This is trivial.

\eqref{s2fpd02b}$\implies$\eqref{s2fpd02c}.
Assume that $\s_R(X)\simeq\shift^j_{} R$, which implies $j = \inf (\s_R(X))$. 
Use Corollary~\ref{symm05b}\eqref{symm05bitem3} to replace $X$ with
a minimal free resolution in order to assume that $X$ is minimal.
As we have noted before, this implies that 
$\s_R(X)$ is minimal, so the quasiisomorphism
$\s_R(X)\simeq \shift^j_{} R$ implies $\s_R(X)\cong\shift^j_{} R$;
see Fact~\ref{fact0001''}.

For each integer $n$, set $r_n=\rank_R(X_n)$.
Also, set $i = \inf (X)$, and note that Proposition~\ref{symm09} implies that
$j\geq 2i$.  
Write $Q(t)=\sum_{n=0}^{\infty}r_{n-i}t^n$; this 
is a formal power series with nonnegative integer coefficients and constant term $r_i\geq 1$
such that
$P_X^R(t) = t^i Q(t)$. 
Since $\s_R(X) \cong \shift^j R$, equation~\eqref{betti02eq1} can be
written as 
\begin{equation}\label{s2fpd02pfeq01}
t^j =\textstyle
\frac{1}{2}\left[(t^iQ(t))^2+(-t^2)^iQ(-t^2)\right]=
\frac{1}{2}t^{2i}\left[Q(t)^2+(-1)^iQ(-t^2)\right].
\end{equation}

Case 1: $j = 2i$.  
In this case,
equation~\eqref{s2fpd02pfeq01} then reads as 
$$t^{2i} = \textstyle
\frac{1}{2}t^{2i}\left[Q(t)^2+(-1)^iQ(-t^2)\right]
$$
and so 
$2 = Q(t)^2+(-1)^i Q(-t^2)$.
Lemma~\ref{poinc01} implies that 
$$Q(t)=
\begin{cases}
1 & \text{if $i$ is even} \\
2 & \text{if $i$ is odd.}
\end{cases}
$$
When $i$ is even, this translates to $P^R_X(t)=t^i$ and so $X\cong\shift^iR=\shift^{2n}R$
where $n=i/2$.
When $i$ is odd, we have $P^R_X(t)=2$ and so 
$X \cong\shift^iR^2 \cong\shift^{2n+1}R\oplus\shift^{2n+1}R$
where $n=(i-1)/2$.

Case 2: $j>2i$.
In this case, Proposition~\ref{symm09} implies that $i$ is odd,
and equation~\eqref{s2fpd02pfeq01} translates as
\begin{align}
2t^{j-2i}
& = Q(t)^2 - Q(-t^2)\notag \\
2t^{j-2i}
&=(r_i^2 - r_i)+2r_{i+1}r_it+(2r_{i+2}r_i + r_{i+1}^2 +
r_{i+1})t^2+\cdots. \label{s2fpd02pfeq04}
\end{align}
Since $j>2i$, we equate coefficients in degree $0$ to find 
$0 = r_i^2 - r_i$, and so $r_i = 1$. Thus, equation~\eqref{s2fpd02pfeq04} reads as
\begin{equation}
2t^{j-2i}=2r_{i+1}t+(2r_{i+2} + r_{i+1}^2 +
r_{i+1})t^2+\cdots. \label{s2fpd02pfeq03}
\end{equation}

We claim that $j>2i+1$. Indeed, supposing that $j\leq 2i+1$,
our assumption $j>2i$ implies
$j = 2i+1$. Equating degree $1$ coefficients in equation~\eqref{s2fpd02pfeq03} yields
$r_{i+1}=1$. 
The coefficients in degree 2 show that 
$$0 = 2r_{i+2}r_i + r_{i+1}^2 +
r_{i+1} = 2r_{i+2} + 2.$$  
Hence $r_{i+2} = -1$, which is a
contradiction.  

Since we have $j > 2i +1$, the
degree $1$ coefficients in equation~\eqref{s2fpd02pfeq03} imply 
$r_{i+1} = 0$.  
It follows that 
\begin{equation}
X\cong \shift^i R \oplus Y
\label{s2fpd02pfeq02}
\end{equation}
where $Y$ is a bounded-below minimal complex of 
finitely generated free $R$-modules such that $Y_n=0$ for all $n<i+2$. 
With the isomorphism in~\eqref{s2fpd02pfeq02},
Proposition~\ref{s2sum01} gives the second isomorphism in the next sequence
$$ \shift^jR\cong\s_R(X) \cong \s_R(\shift^i R) \oplus \left[(\shift^i R) \otimes_R Y\right] \oplus
\s_R(Y) \cong  \shift^i  Y \oplus
\s_R(Y).$$
The final isomorphism comes from Example~\ref{symm045}
since  $i$ is odd. In particular, it follows that $Y\not\simeq 0$.
The complex $\shift^jR$ is indecomposable because $R$ is local,
so the displayed sequence implies that $\s_R(Y) =0$ and $\shift^iY\simeq\shift^jR$.
Because of the conditions $\s_R(Y) =0$ and $Y\not\simeq 0$,
Theorem~\ref{symm07''} implies that  $Y\simeq
\shift^{2m+1} R$ for some $m$.  Hence, the isomorphism in~\eqref{s2fpd02pfeq02}
reads as 
$X\cong \shift^{2n+1} R \oplus \shift^{2m+1} R$
where $n=(i-1)/2$, as desired.
\end{proof}

\section{Examples}\label{sec05}

We begin this section with three explicit computations of 
the complexes $\s_R(X)$ and $\sw_R(X)$ and their homologies. 
As a consequence, we show that Buchbaum and Eisenbud's construction
differs from those in~\cite{dold:hnafa, tchernev:frpf}.
We also provide examples
showing the need for certain hypotheses in the results of the previous sections.

\begin{ex} \label{symm03}
Fix an element $x\in R$ and let $K$ denote the Koszul complex $K^R(x)$
which has the following form,
where the basis is listed in each degree
\begin{equation} \label{K02} 
K=\qquad
0\to
\hspace{-2mm}\underbrace{R}_{e_1}\hspace{-2mm}
\xra{(x)}
\hspace{-1mm}\underbrace{R}_{e_0} \hspace{-1mm}
\to 0.
\end{equation}
The tensor product $K\otimes_R K$ has the  form
$$K\otimes_RK=\qquad
0\to 
\hspace{-2mm}\underbrace{R}_{e_1\otimes e_1}\hspace{-2mm}
\xra{\left(\begin{smallmatrix} x \\ -x\end{smallmatrix}\right)}
\hspace{-2mm}\underbrace{R^2}_{\begin{smallmatrix}e_0\otimes e_1 \\
e_1\otimes e_0\end{smallmatrix}}\hspace{-2mm}
\xra{\left(\begin{smallmatrix} x & x\end{smallmatrix}\right)}
\hspace{-2mm}\underbrace{R}_{e_0\otimes e_0}\hspace{-2mm}
 \to 0.
$$
Using this representation, the exact sequence in~\eqref{symm04item1}
has the  form
$$
\xymatrix@R=1em{
0 \ar[r] & \ker(\alf X) \ar[r] & K\otimes_RK \ar[r]^{\alf K}
& K\otimes_RK \ar[r] & \sw_R(K) \ar[r]  & 0 \\
& 0 \ar[dd] & 0 \ar[dd] & 0 \ar[dd] & 0 \ar[dd] \\ \\
0 \ar[r]
& \ann_R(2) \ar[r] \ar[dd]_{(x)}
& R \ar[r]^{(2)} \ar[dd]_{\left(\begin{smallmatrix} x \\ -x\end{smallmatrix}\right)}
& R \ar[r] \ar[dd]_{\left(\begin{smallmatrix} x \\ -x\end{smallmatrix}\right)}
& R/(2) \ar[r] \ar[dd]_{(0)}
& 0 \\ \\
0 \ar[r]
& R \ar[r]^{\left(\begin{smallmatrix} 1 \\ 1\end{smallmatrix}\right)} \ar[dd]_{(2x)}
& R^2 \ar[r]^<<<<<<<{\left(\begin{smallmatrix} 1 & -1 \\ -1 & 1\end{smallmatrix}\right)}
\ar[dd]_{\left(\begin{smallmatrix} x & x\end{smallmatrix}\right)}
& R^2 \ar[r]^{\left(\begin{smallmatrix} 1 & 1\end{smallmatrix}\right)}
\ar[dd]_{\left(\begin{smallmatrix} x & x\end{smallmatrix}\right)}
& R \ar[r] \ar[dd]_{(x)}
& 0 \\ \\
0 \ar[r]
& R \ar[r]^{(1)} \ar[dd]
& R \ar[r]^{(0)} \ar[dd]
& R \ar[r]^{(1)} \ar[dd]
& R \ar[r] \ar[dd]
& 0 \\ \\
& 0 & 0 & 0 & 0.
}
$$
From the rightmost column of this diagram, we have
\begin{align*}
\HH_2(\sw_R(K))&\cong R/(2)
&\HH_1(\sw_R(K))&\cong \ann_R(x)
&\HH_0(\sw_R(K))&\cong R/(x)
\end{align*}
and $\HH_i(\sw_R(K))=0$ when $i\notin \{0,1,2\}$.

A similar computation shows that 
$$\s_R(K)=\qquad
0\to R\xra{(x)}R\to 0$$
and thus
\begin{align*}
\HH_1(\s_R(K))&\cong \ann_R(x)
&\HH_0(\s_R(K))&\cong R/(x)
\end{align*}
and $\HH_i(\s_R(K))=0$ when $i\notin \{0,1\}$.
\end{ex}

\begin{ex} \label{koszul01}
Fix elements $x,y\in R$  and let $K$ denote
the Koszul complex $K^R(x,y)$ which has the following form,
where the ordered basis is listed in each degree
\begin{equation} \label{K01} 
K=\qquad
0\to
\hspace{-2mm}\underbrace{R}_{e_2}\hspace{-2mm}
\xra{\left(\begin{smallmatrix} y \\ -x\end{smallmatrix}\right)}
\hspace{-1mm}\underbrace{R^2}_{\begin{smallmatrix}e_{11} \\ e_{12}\end{smallmatrix}}
\hspace{-1mm}
\xra{\left(\begin{smallmatrix} x & y \end{smallmatrix}\right)}
\hspace{-1mm}\underbrace{R}_{e_0} \hspace{-1mm}
\to 0.
\end{equation}
Using the same format, the complex $K\otimes_R K$ has the  form
$$K\otimes_R K= \qquad
0\to
\hspace{-1mm}\underbrace{R}_{e_2\otimes e_2}\hspace{-1mm}
\xra{\partial^{K\otimes_R K}_4}
\hspace{-2mm}\underbrace{R^4}_{\begin{smallmatrix}e_{2}\otimes e_{11}\\
e_{2}\otimes e_{12}\\e_{11}\otimes e_{2}\\e_{12}\otimes e_{2}\end{smallmatrix}}\hspace{-2mm}
\xra{\partial^{K\otimes_R K}_3}
\hspace{-2mm}\underbrace{R^6}_{\begin{smallmatrix}e_{2}\otimes e_{0}\\
e_{11}\otimes e_{11}\\e_{11}\otimes e_{12}\\
e_{12}\otimes e_{11}\\e_{12}\otimes e_{12}\\e_{0}\otimes e_{2}\end{smallmatrix}} \hspace{-2mm}
\xra{\partial^{K\otimes_R K}_2}
\hspace{-2mm}\underbrace{R^4}_{\begin{smallmatrix}e_{11}\otimes e_{0}\\
e_{12}\otimes e_{0}\\e_{0}\otimes e_{11}\\e_{0}\otimes e_{12}\end{smallmatrix}}\hspace{-2mm}
\xra{\partial^{K\otimes_R K}_1}
\hspace{-2mm}\underbrace{R}_{e_0\otimes e_0} \hspace{-2mm}
\to 0
$$
with differentials given by the following matrices:
\begin{align*}
\partial^{K\otimes_R K}_4
&=\left(\begin{smallmatrix} y \\ -x \\ y \\ -x \end{smallmatrix}\right)
&\partial^{K\otimes_R K}_3
&=\left(\begin{smallmatrix} x & y & 0 & 0 \\ y & 0 & -y & 0 \\ 0 & y & x & 0 \\
-x & 0 & 0 & -y \\ 0 & -x & 0 & x \\ 0 & 0 & x & y \end{smallmatrix}\right) \\
\partial^{K\otimes_R K}_2
&=\left(\begin{smallmatrix} y & -x & -y & 0 & 0 & 0 \\
-x & 0 & 0 & -x & -y & 0 \\ 0 & x & 0 & y & 0 & y \\
0 & 0 & x & 0 & y & -x \end{smallmatrix}\right)
& \partial^{K\otimes_R K}_1
&=(x \,\,\,\,  y \,\,\,\,  x \,\,\,\,  y).
\end{align*}
Under the same bases,
the morphism $\alf K\colon K\otimes_R K\to K\otimes_R K$ is described
by the following matrices: 
\begin{align*}
 \alf K_3
&=\left(\begin{smallmatrix} 1 & 0 & -1 & 0 \\ 0 & 1 & 0 & -1 \\
-1 & 0 & 1 & 0 \\ 0 & -1 & 0 & 1 \end{smallmatrix}\right) 
&\alf K_2
&=\left(\begin{smallmatrix}
1 & 0 & 0 & 0 & 0 & -1 \\
0 & 2 & 0 & 0 & 0 & 0 \\
0 & 0 & 1 & 1 & 0 & 0 \\
0 & 0 & 1 & 1 & 0 & 0 \\
0 & 0 & 0 & 0 & 2 & 0 \\
-1 & 0 & 0 & 0 & 0 & 1
\end{smallmatrix}\right)
\\
\alf K_1
&=\left(\begin{smallmatrix}  1 & 0 & -1 & 0 \\ 0 & 1 & 0 & -1 \\
-1 & 0 & 1 & 0 \\ 0 & -1 & 0 & 1
\end{smallmatrix}\right)
&\alf K_4
&=(0)
=\alf K_0.
\end{align*}
As in Example~\ref{symm03}, it follows that $\s_R(K)$ has the form
$$\s_R(K)= \qquad
0\to \hspace{-1mm}\underbrace{R}_{f_4}\hspace{-1mm}
\xra{\partial^{\s_R(K)}_4}
\hspace{-1mm}\underbrace{R^2}_{\begin{smallmatrix}f_{31}\\
f_{32}\end{smallmatrix}} \hspace{-1mm} \xra{\partial^{\s_R(K)}_3}
\hspace{-1mm}\underbrace{R^2}_{\begin{smallmatrix}f_{21}\\
f_{22}\end{smallmatrix}} \hspace{-1mm} \xra{\partial^{\s_R(K)}_2}
\hspace{-1mm}\underbrace{R^2}_{\begin{smallmatrix}f_{11}\\
f_{12}\end{smallmatrix}} \hspace{-1mm} \xra{\partial^{\s_R(K)}_1}
\hspace{-1mm}\underbrace{R}_{f_0} \hspace{-1mm} \to 0
$$
where the basis vectors are described as
\begin{align*}
f_4
&=\ol{e_2\otimes e_2} 
&f_{31}
&=\ol{e_{2}\otimes e_{11}}=\ol{e_{11}\otimes e_{2}} \\
f_{32}
&=\ol{e_{2}\otimes e_{12}}=\ol{e_{12}\otimes e_{2}} 
&f_{21}
&=\ol{e_{2}\otimes e_{0}}=\ol{e_{0}\otimes e_{2}} \\
f_{22}
&=\ol{e_{11}\otimes e_{12}}=-\ol{e_{12}\otimes e_{11}} 
&f_{11}
&=\ol{e_{11}\otimes e_{0}}=\ol{e_{0}\otimes e_{11}} \\
f_{12}
&=\ol{e_{12}\otimes e_{0}}=\ol{e_{0}\otimes e_{12}} 
&f_0
&=\ol{e_0\otimes e_0}.
\end{align*}
(Note also that $\ol{e_{11}\otimes e_{11}}=0=\ol{e_{12}\otimes e_{12}}$.)
Under these bases, the differentials $\partial^{\s_R(K)}_n$ are described
by the following matrices:
\begin{equation}
\begin{split} \label{s2K01}
\begin{aligned}
\partial^{\s_R(K)}_4
&=\begin{pmatrix} 2y \\ -2x \end{pmatrix} \\
\partial^{\s_R(K)}_2
&=\begin{pmatrix} y & -y \\ -x & x \end{pmatrix}
\end{aligned}
\qquad \qquad \qquad
\begin{aligned}
\partial^{\s_R(K)}_3
&=\begin{pmatrix} x & y \\ x & y \end{pmatrix}    \\
\partial^{\s_R(K)}_1
&=\begin{pmatrix} x & y  \end{pmatrix}.
\end{aligned}
\end{split}
\end{equation}
Similar computations show that $\sw_R(K)\cong\s_R(K)\oplus\shift^2(R/(2))^2$.
\end{ex}

\begin{ex} \label{koszul01'}
Let $x,y\in R$  be an $R$-regular sequence and
continue with the notation of Example~\ref{koszul01}.
We verify the following isomorphisms:
\begin{align*}
\HH_0(S^0_R(K))
&\cong\HH_2(\s_R(K))
\cong R/(x,y) 
&\HH_1(\s_R(K))
&=0\\
\HH_3(\s_R(K))
&\cong R/(2)
&\HH_4(\s_R(K))
&\cong \ann_R(2).
\end{align*}

The computation of $\HH_0(\s_R(K))$ follows from the
description of $\partial^{\s_R(K)}_1$ in~\eqref{s2K01}.

For $\HH_1(\s_R(K))$, the second equality in the following sequence
comes from the exactness of $K$ in degree 1
$$\ker\left(\partial^{\s_R(K)}_1 \right)=\ker \left(\partial^K_1 \right)=\im \left(\partial^K_2 \right)
=\Span_R\left\{\begin{pmatrix} y \\ -x \end{pmatrix}\right\}
=\im \left(\partial^{\s_R(K)}_2\right)$$
and the others come from the
descriptions of $K$ and $\s_R(K)$ in~\eqref{K01} and~\eqref{s2K01}.

For $\HH_2(\s_R(K))$,
use the fact that $x$ is $R$-regular to check the first equality
in the next display;
the others follow from~\eqref{s2K01}.
\begin{align*}
\ker \left(\partial^{\s_R(K)}_2 \right)
&=\Span_R\left\{\begin{pmatrix} 1 \\ 1 \end{pmatrix}\right\} \\
\im \left(\partial^{\s_R(K)}_3 \right)
&=\Span_R\left\{\begin{pmatrix} x \\ x \end{pmatrix},
\begin{pmatrix} y \\ y \end{pmatrix}\right\}
=(x,y)\Span_R\left\{\begin{pmatrix} 1 \\ 1 \end{pmatrix}\right\}
\end{align*}
The isomorphism
$\HH_2(\s_R(K))
\cong R/(x,y)$ now follows.

For $\HH_3(\s_R(K))$, the second equality in the following sequence
comes from the exactness of $K$ in degree 1
\begin{align*}
\ker\left(\partial^{\s_R(K)}_3 \right)
&
=\ker \left(\partial^K_1\right) 
=\im \left(\partial^K_2 \right)
=\Span_R\left\{\begin{pmatrix} y \\ -x \end{pmatrix}\right\}
\\
\im \left(\partial^{\s_R(K)}_4\right)
&=(2)\Span_R\left\{\begin{pmatrix} y \\ -x \end{pmatrix}\right\}
\end{align*}
and the others come from the
descriptions of $K$ and $\s_R(K)$ in~\eqref{K01} and~\eqref{s2K01}.
The isomorphism
$\HH_3(\s_R(K))
\cong R/(2)$ now follows.

Similarly, for $\HH_4(\s_R(K))$, we have
$$\HH_4(\s_R(K))=\ker\left(\partial^{\s_R(K)}_4 \right)
=(\ker\left(\partial^{K}_2 \right)\colon 2)=(0:_R 2)=\ann_R(2).$$
This completes the example.
\end{ex}

As a first consequence of the previous computations, we next observe that 
$\s_R(X)$ is generally not isomorphic to Dold and Puppe's~\cite{dold:hnafa}
construction $\cat{D}_{\s}(X)$ and not isomorphic to Tchernev and Weyman's~\cite{tchernev:frpf}
construction $\catc_{\s}(X)$. 

\begin{ex} \label{ex0401}
Assume that 2 is a unit in $R$.
Fix an element $x\in R$ and let $K$ denote the Koszul complex $K^R(x)$.
Example~\ref{symm03} yields the following computation of $\s_R(K)$
$$\xymatrix@R=1em{
\hspace{17mm}\s_R(K)=\hspace{-5mm}
&
& 0\ar[r] 
& R\ar[r]^-{x} 
& R\ar[r] 
& 0 \\
\cat{D}_{\s}(K)\cong \catc_{\s}(K)=\hspace{-5mm}
& 0\ar[r] 
& R\ar[r]^-{\left(\begin{smallmatrix}1 \\ -x\end{smallmatrix}\right)} 
& R^2 \ar[r]^-{(x^2\,\,\,\, x)} 
& R\ar[r] 
& 0.
}$$
The fact that $\cat{D}_{\s}(K)$ and $\catc_{\s}(K)$ have the displayed 
form can be deduced from~\cite[(11.2) and (14.4)]{tchernev:frpf};
the maps were computed for us by Tchernev.
In particular, in this case we have $\cat{D}_{\s}(K)\cong \catc_{\s}(K)\not\cong\s_R(K)$.

More generally, if we have
$$X=\quad 0\to R^m\to R^n\to 0$$
then Corollary~\ref{betti011} and~\cite[(11.2) and (14.4)]{tchernev:frpf} yield
$$\xymatrix@R=1em{
\hspace{17mm}\s_R(X)=\hspace{-5mm}
& 0\ar[r]
& R^{\binom{m}{2}}_{} \ar[r]
& R^{mn}_{} \ar[r]
& R^{\binom{n+1}{2}}_{} \ar[r]
& 0 \\
\cat{D}_{\s}(X)\cong \catc_{\s}(X)=\hspace{-5mm}
& 0 \ar[r]
& R^{m^2}_{} \ar[r]
& R^{\binom{m+1}{2}+mn}_{} \ar[r]
& R^{\binom{n+1}{2}}_{} \ar[r]
& 0.
}$$
Hence, we have $\catc_{\s}(X)\cong\s_R(X)$ if and only if $m=0$, i.e.,
if and only if $X\cong R^n$.
\end{ex}

We next show why we need to assume 
that $X$ and $Y$ are bounded-below complexes
of projective $R$-modules in Corollary~\ref{symm05b}.
It also shows that 
$\s_R(X)$ can have nontrivial homology, even when $X$ is a minimal free 
resolution of a module of finite projective dimension.

\begin{ex} \label{koszul01''}
Let $x,y\in R$  be an $R$-regular sequence and
continue with the notation of Example~\ref{koszul01}.
The computations  in Example~\ref{koszul01'} 
show that $\HH_2(\s_R(K))
\cong R/(x,y)\neq 0=\HH_2(\s_R(R/(x,y)))$,
and so  $\s_R(K)\not\simeq\s_R(R/(x,y))$ even though
$K\simeq R/(x,y)$.
\end{ex}

The next example shows why we need to assume that 2 is a unit in
$R$ for Theorem~\ref{symm05a} and
Corollaries~\ref{symm05b} and~\ref{loc01'}.

\begin{ex} \label{symm05c}
Assume that 2 is not a unit in $R$ and let $K$ denote the Koszul
complex $K^R(1,1)$. Then $K$ is split exact, so the zero map $z\colon
K\to K$ is a homotopy equivalence, it is homotopic to $\id_K$, and
it is a quasiisomorphism.  Example~\ref{koszul01} 
shows that $\HH_3(\s_R(K))=R/(2)\neq 0$. On the other hand, the
morhpism $\s_R(z)\colon \s_R(K)\to\s_R(K)$ is the zero morphism, so
the nonvanishing of $\HH_2(\s_R(K))$ implies that $\s_R(z)$ is not a
quasiisomorphism.  It follows that $\s_R(z)$ is neither a homotopy
equivalence nor homotopic to $\id_{\s_R(K)}$.  This shows 
why we must assume that 2 is a unit in $R$ for
Theorem~\ref{symm05a} and
Corollary~\ref{symm05b}\eqref{symm05bitem2}. For
Corollary~\ref{symm05b}\eqref{symm05bitem3} simply note that
$K\simeq 0$ and $\s_R(K)\not\simeq 0\simeq\s_R(0)$.
For  Corollary~\ref{loc01'}, note that this shows that $\supp_R(\s_R(K))\neq
\emptyset=\supp_R(K)$.
\end{ex}

Our next example shows that the functors $\sw_R(-)$ and $\s_R(-)$ are not additive, 
even when 2 is a unit in $R$ and 
we restrict to bounded complexes of finite rank free $R$-modules.

\begin{ex} \label{notadd01}
Let $X$ and $Y$ be nonzero $R$-complexes. Consider the natural
surjections and injections
\begin{align*}
X\oplus Y\xra{\tau_1}X\xra{\epsilon_1}X\oplus Y
&&X\oplus Y\xra{\tau_2}Y\xra{\epsilon_2}X\oplus Y
\end{align*}
and set $f_i=\epsilon_i\tau_i\colon X\oplus Y\to X\oplus Y$.
The equality $f_1+f_2=\id_{X\oplus Y}$ is immediate.

We claim that $\sw_R(f_1+f_2)\neq\sw_R(f_1)+\sw_R(f_2)$.
To see this, first note that the equalities
$\sw_R(f_1+f_2)=\sw_R(\id_{X\oplus Y})=\id_{\sw_R(X\oplus Y)}$
show that it suffices to verify
$\sw_R(f_1)+\sw_R(f_2)\neq\id_{\sw_R(X\oplus Y)}$.
One checks that there is a commutative diagram
$$
\xymatrix{
(X\oplus Y)\otimes_R (X\oplus Y) \ar[r]^-{\cong} \ar[dd]_{f_1\otimes_R f_1}
& (X\otimes_RX) \oplus (X\otimes_RY) \oplus (Y\otimes_RX) \oplus (Y\otimes_RY)
\ar[dd]^{\left(\begin{smallmatrix}
\id_{X\otimes_R X} & 0 & 0 & 0 \\
0 & 0 & 0 & 0 \\
0 & 0 & 0 & 0 \\
0 & 0 & 0 & 0
\end{smallmatrix}\right)} \\
\\
(X\oplus Y)\otimes_R (X\oplus Y) \ar[r]^-{\cong}
& (X\otimes_RX) \oplus (X\otimes_RY) \oplus (Y\otimes_RX) \oplus (Y\otimes_RY)
}
$$
wherein the horizontal maps are the natural distributivity isomorphisms.
The proof of Proposition~\ref{s2sum01} yields another commutative diagram
$$
\xymatrix{
\sw_R(X\oplus Y) \ar[r]^-{\cong} \ar[dd]_{\sw_R(f_1)}
& \sw_R(X) \oplus (X\otimes_RY)  \oplus \sw_R(Y)
\ar[dd]^{\left(\begin{smallmatrix}
\id_{\sw_R(X)} & 0 & 0 \\
0 & 0 & 0 \\
0 & 0 & 0
\end{smallmatrix}\right)} \\
\\
\sw_R(X\oplus Y) \ar[r]^-{\cong}
& \sw_R(X) \oplus (X\otimes_RY)  \oplus \sw_R(Y).
}
$$
Similarly, there is another commutative diagram
$$
\xymatrix{
\sw_R(X\oplus Y) \ar[r]^-{\cong} \ar[dd]_{\sw_R(f_2)}
& \sw_R(X) \oplus (X\otimes_RY)  \oplus \sw_R(Y)
\ar[dd]^{\left(\begin{smallmatrix}
0 & 0 & 0 \\
0 & 0 & 0 \\
0 & 0 & \id_{\sw_R(Y)}
\end{smallmatrix}\right)} \\
\\
\sw_R(X\oplus Y) \ar[r]^-{\cong}
& \sw_R(X) \oplus (X\otimes_RY)  \oplus \sw_R(Y).
}
$$
This implies
that $\sw_R(f_1)+\sw_R(f_2)$ is equivalent to the morphism
$$ \sw_R(X) \oplus (X\otimes_RY)  \oplus \sw_R(Y)
\xra{\left(\begin{smallmatrix}
\id_{\sw_R(X)} & 0 & 0 \\
0 & 0 & 0 \\
0 & 0 & \id_{\sw_R(Y)}
\end{smallmatrix}\right)} \sw_R(X) \oplus (X\otimes_RY)  \oplus \sw_R(Y)$$
and so cannot equal $\id_{\sw_R(X\oplus Y)}$.

Similarly, we have
$\s_R(f_1+f_2)=\s_R(\id_{X\oplus Y})=\id_{\s_R(X\oplus Y)}\neq\s_R(f_1)+\s_R(f_2)$.
\end{ex}

Our final example shows that one needs to be 
careful about removing the local hypotheses from the results
of Section~\ref{sec03}. Specifically, it shows that,
without the local hypothesis, the implication 
\eqref{symm07item1a}$\implies$\eqref{symm07item1d}
fails in Theorem~\ref{symm07}.

\begin{ex} \label{ex0002}
Let $K$ and $L$ be fields, and set $R=K\times L$. 
The prime ideals of $R$ are all maximal, and they are precisely the ideals
$\m=K\times 0$ and $\n=0\times L$. Furthermore, we have
$R_{\m}\cong L$ and $R_{\n}\cong K$.
Assume that $\ch(K)\neq 2$ and $\ch(L)\neq 2$, so that 2 is a unit in $R$.

First, consider the complex
$Y=(K\times 0)\oplus\shift^2(0\times L)$.
Then $Y$ is a bounded-below complex of
finitely generated projective $R$-modules such that 
$Y_{\m}\cong \shift^2L\cong \shift^2R_{\m}$ and $Y_{\n}\cong K\cong R_{\n}$.
Hence, Remark~\ref{disc0002} implies that the surjection $\ppro Y\colon Y\otimes_R Y\to \s_R(Y)$ 
is a quasiisomorphism. However, the fact that $Y$ has nonzero homology in
degrees 2 and 0 implies that
$Y\not\simeq 0$ and $Y\not\simeq\shift^{2t} R$ for each integer $t$.

Next we provide an example of a bounded-below complex $X$ of
finitely generated \emph{free} $R$-modules with the same behavior.
The following complex describes a free resolution $F$ of $K\times 0$
$$\cdots\xra{(e)}R\xra{(f)}R\xra{(e)}R\xra{(f)}\cdots\xra{(f)}R\to 0$$
where $e=(1,0)\in R$ and
$f=(0,1)\in R$. An $R$-free resolution $G$ for $0\times L$ is constructed similarly.
The complex $X=F\oplus \shift^2 G$ yields a
degreewise-finite $R$-free resolution of $g\colon X\res Y$. 
Corollary~\ref{symm05b}\eqref{symm05bitem2} implies that
$\s_R(g)$ is a quasiisomorphism. Hence, the next commutative diagram shows that 
the surjection $\ppro X\colon X\otimes_R X\to \s_R(X)$ is  also a quasiisomorphism.
$$\xymatrix{
X\otimes_RX \ar[r]^{\ppro X} \ar[d]_{\simeq}^{g\otimes g} 
& \s_R(X) \ar[d]_{\simeq}^{\s(g)} \\
Y\otimes_RY \ar[r]^{\ppro Y}_{\simeq} 
& \s_R(Y) 
}$$
However, we have $X\simeq Y$, and so 
$X\not\simeq 0$ and $X\not\simeq\shift^{2t} R$ for each integer $t$.
\end{ex}

\section*{Acknowledgments}
We are grateful to W.\ Frank Moore and Alex Tchernev for
helpful conversations about the existing literature on this subject.
We are grateful to the referee for thoughtful comments.

\providecommand{\bysame}{\leavevmode\hbox to3em{\hrulefill}\thinspace}
\providecommand{\MR}{\relax\ifhmode\unskip\space\fi MR }
\providecommand{\MRhref}[2]{%
  \href{http://www.ams.org/mathscinet-getitem?mr=#1}{#2}
}
\providecommand{\href}[2]{#2}

\end{document}